\documentclass[a4paper]{amsart}

\usepackage{amsmath,amssymb,amsthm,amscd}
\usepackage[matrix,arrow,curve,tips]{xy}
            \SelectTips{cm}{}

\usepackage{mathrsfs}
\usepackage{stmaryrd}
\usepackage{mathabx}

\newtheorem{dfn}{Definition}[section]

\newtheorem{prop}[dfn]{Proposition}
\newtheorem{theo}[dfn]{Theorem}
\newtheorem{ex}[dfn]{Example}

\newtheorem{rem}[dfn]{Remark}

\newcommand{\RR}{\mathbb{R}}
\newcommand{\CC}{\mathbb{C}}
\newcommand{\DD}{\mathbb{D}}

\newcommand{\NN}{\mathbb{N}}

\newcommand{\cE}{\mathcal{E}}

\newcommand{\Hom}{\text{Hom}}
\newcommand{\Lin}{\text{Lin}}

\newcommand{\src}{\mathord{\mathrm{s}}}
\newcommand{\trg}{\mathord{\mathrm{t}}}

\newcommand{\com}{\mathbin{{\scriptstyle \circ }}}
\newcommand{\ten}{\mathbin{\otimes}}

\newcommand{\id}{\mathord{\mathrm{id}}}

\newcommand{\supp}{\mathord{\mathrm{supp}}}
\newcommand{\uni}{\mathord{\mathrm{uni}}}
\newcommand{\inv}{\mathord{\mathrm{inv}}}
\newcommand{\mlt}{\mathord{\mathrm{mlt}}}
\newcommand{\Dirac}{\mathord{\mathrm{Dirac}}}
\newcommand{\Spec}{\mathord{\mathrm{Spec}}}
\newcommand{\Sp}{\mathord{\mathrm{Sp}}}

\newcommand{\cu}{\mathord{\epsilon}}
\newcommand{\cm}{\mathord{\Delta}}
\newcommand{\C}{\mathord{\mathcal{C}^{\infty}}}
\newcommand{\Cc}{\mathord{\mathcal{C}^{\infty}_{c}}}

\newcommand{\Esp}{\mathord{\mathcal{E}_{\mathit{sp}}}}

\newcommand{\Bsp}{\mathord{\mathcal{B}_{\mathit{sp}}}}
\newcommand{\Gsp}{\mathord{\mathcal{G}_{\mathit{sp}}}}
\newcommand{\AGsp}{\mathord{\mathcal{AG}_{\mathit{sp}}}}

\newcommand{\pisp}{\mathord{\pi_{\mathit{sp}}}}
\newcommand{\Lt}{\mathord{{\mathrm{L}}}}
\newcommand{\GG}{\mathscr{G}}

\newcommand{\Nd}{N^{\scriptscriptstyle\#}}
\newcommand{\Hd}{H^{\scriptscriptstyle\#}}
\newcommand{\pid}{\pi^{\scriptscriptstyle\#}}

\title[]
      {Locally convex bialgebroid of an action Lie groupoid}
      
\author{J. Kali\v{s}nik}
\address{Department of Mathematics, University of Ljubljana,
         Jadranska 19, 1000 Ljubljana, Slovenia;
         Institute of Mathematics, Physics and Mechanics,
         University of Ljubljana, Jadranska 19,
         1000 Ljubljana, Slovenia}
\email{jure.kalisnik@fmf.uni-lj.si}

\thanks{This research was supported by research grants J1-1690, N1-0137, N1-0237 and research program Analysis and Geometry P1-0291 from 
the Slovenian Research Agency ARRS}


\begin{document}

\begin{abstract}
Action Lie groupoids are used to model spaces of orbits of actions of Lie groups on manifolds. For each such action groupoid $M\rtimes H$ we
construct a locally convex bialgebroid $\Dirac(M\rtimes H)$ with an antipode over $\Cc(M)$, from which the groupoid $M\rtimes H$
can be reconstructed as its spectral action Lie groupoid $\AGsp(\Dirac(M\rtimes H))$.
\end{abstract}

\maketitle

\section{Introduction}

Our motivation for this paper originates from the Gelfand-Naimark theorem. To any locally compact, Hausdorff
topological space $X$ one assigns the $C^{*}$-algebra $C_{0}(X)$ of all continuous functions on $X$ that vanish at infinity.
The set $\Spec(C_{0}(X))$ of all characters on $C_{0}(X)$ can be equipped with the weak-$^{*}$ topology so that it
becomes a locally compact Hausdorff space. The Gelfand-Naimark theorem then says that the map $\Phi^{\text{lc}}_{X}:X\to\Spec(C_{0}(X))$,
which assigns to a point $x\in X$ the evaluation $\delta_{x}$ at $x$, is a homeomorphism. 

We wish to obtain a similar result for the class of geometric spaces which can be represented by Lie groupoids \cite{CSWe99,Mac05,MoMr03,MoMr05}.
These spaces include orbifolds, spaces of leaves of foliations and spaces of orbits of Lie groups actions. In \cite{Mrc07_1} a result in the
spirit of the Gelfand-Naimark theorem was constructed for the class of \'{e}tale Lie groupoids, which can be used to model orbifolds and
spaces of leaves of foliations. To any \'{e}tale Lie groupoid $\GG$ one assigns the Hopf algebroid $\Cc(\GG)$
of smooth compactly supported functions on $\GG$ (if $\GG$ is not Hausdorff, one needs to be careful with the definition). As
an algebra $\Cc(\GG)$ coincides with the Connes convolution algebra in the Hausdorff case \cite{Con82}, while the coalgebra
structure is basically induced from the sheaf \cite{Mrc07_2}, corresponding to the target map $\trg:\GG\to M$ of $\GG$. Finally, the antipode
on $\Cc(\GG)$ is induced by the inverse map of $\GG$. For each such Hopf algebroid $\Cc(\GG)$ one can construct
the spectral \'{e}tale Lie groupoid $\Gsp(\Cc(\GG))$ so that there is a natural isomorphism $\Phi^{\text{egr}}_{\GG}:\GG\to\Gsp(\Cc(\GG))$
of Lie groupoids. Similar ideas were used in \cite{KaMr13} to extend these results to the semi-direct products
of \'{e}tale Lie groupoids and bundles of Lie groups.

Structure maps of an \'{e}tale Lie groupoid $\GG$ are local diffeomorphisms so, in particular, the fibres of the target map are discrete.
It is therefore enough to reconstruct the fibres of such a groupoid just as sets, which can be done by utilizing the coalgebra
structure on $\Cc(\GG)$. However, in the case of a general groupoid one needs some additional information, which enables us
to recover the topology along the fibres. Let us explain the main idea on a simple example. If $\Gamma$ is a discrete
group, then $\Cc(\Gamma)$ is just the group Hopf algebra of $\Gamma$. Elements of $\Gamma$ correspond to grouplike elements of $\Cc(\Gamma)$,
while the multiplication and inverse of $\Gamma$ are encoded in multiplication and antipode of $\Cc(\Gamma)$. If we now replace
$\Gamma$ with a non-discrete Lie group $H$, one can still define its group Hopf algebra and reconstruct $H$, but only as a group. To recover
the topology and the smooth structure of $H$ we need some additional structure. One way to solve this problem is to identify the
group Hopf algebra of $H$ with the space $\Dirac(H)$ of distributions on $H$ which is spanned by Dirac distributions. The
space $\Dirac(H)$ is a subspace of the space $\cE'(H)$ of compactly supported distributions on $H$. If we equip $\Dirac(H)$
with the induced strong topology from $\cE'(H)$, the following two things happen. The group $H$ is naturally homeomorphic to the
space of Dirac distributions, which are precisely the grouplike elements of the Hopf algebra $\Dirac(H)$. On the other
hand, since $\Dirac(H)$ is dense in $\cE'(H)$, the strong dual $\Dirac(H)'$ is isomorphic to $\C(H)$. Now observe that
the space $\C(H)$ is a Fr\'{e}chet algebra, from which $H$ can be reconstructed as a manifold.

Judging by the above example, we are led to consider Hopf algebroids not only as purely algebraic objects, but with some
additional topological structure. The main idea consists of two parts. First of all we assign to a Lie groupoid $\GG$ 
a certain Hopf algebroid, from which the algebraic structure of $\GG$ can be reconstructed. We then equip this 
Hopf algebroid with a suitable locally convex structure, which enables us to recover the topology and smooth 
structure of $\GG$.

In this paper we use this idea on the class of action Lie groupoids, which are used to describe spaces of orbits of Lie groups actions
on manifolds. Each such action Lie groupoid $M\rtimes H$ is isomorphic as a groupoid to an \'{e}tale Lie groupoid
$M\rtimes\Hd$, where $\Hd$ is the group $H$ with discrete topology. We use this identification
to define the Dirac bialgebroid $\Dirac(M\rtimes H)$ of $M\rtimes H$ as a certain subspace of
the space $\cE'_{\trg}(M\rtimes H)$
of $\trg$-transversal distributions on $M\rtimes H$. Transversal distributions on Lie groupoids were studied
in \cite{AnMoYu21,AnSk11,ErYu19,KaMr22,LeMaVa17} and, crucially for our problem, it was shown in \cite{LeMaVa17} that the space
$\cE'_{\trg}(M\rtimes H)$ is a locally convex algebra, if we equip it with the strong topology
of uniform convergence on bounded subsets. With the induced topology $\Dirac(M\rtimes H)$ becomes a locally
convex bialgebroid with an antipode.

The paper is organized as follows. In Section \ref{Section Preliminaries} we recall the basic definitions and known results
that are used in the rest of the paper. In Section \ref{Dirac distributions of constant type} we construct for every
trivial bundle $\pi:M\times N\to M$ the space $\Dirac_{\pi}(M\times N)$ of transversal distributions of constant Dirac type. These
are families of Dirac distributions, supported on constant sections of $\pi$. If the fiber $N$ is discrete,
$\Dirac_{\pi}(M\times N)$ coincides with the $LF$-space $\Cc(M\times N)$, while in general
we show that it is a dense subspace of the space of $\pi$-transversal distributions $\cE'_{\pi}(M\times N)$. 
In Section \ref{Section Spectral bundle of coalgebra of Dirac distributions} we define on $\Dirac_{\pi}(M\times N)$
a structure of a locally convex coalgebra over $\Cc(M)$ and show that its strong $\Cc(M)$-dual is naturally
isomorphic to the Fr\'{e}chet algebra $\C(M\times N)$. The combination of the coalgebra structure and
locally convex topology enables us to reconstruct from $\Dirac_{\pi}(M\times N)$ the bundle $\pi:M\times N\to M$ as
the spectral bundle $\Bsp(\Dirac_{\pi}(M\times N))$. Finally, in Section \ref{Section Locally convex bialgebroid}
we use these results to assign to each action Lie groupoid $M\rtimes H$ its Dirac bialgebroid $\Dirac(M\rtimes H)$.
The space $\Dirac(M\rtimes H)$ is a locally convex bialgebroid with an antipode, which coincides with the locally
convex Hopf algebra $\Dirac(H)$ in the case when $M$ is a point. We then show that the groupoid $M\rtimes H$
can be reconstructed from $\Dirac(M\rtimes H)$ as its spectral action Lie groupoid $\AGsp(\Dirac(M\rtimes H))$

\section{Preliminaries}\label{Section Preliminaries}

In this subsection we will review basic definitions and results that will be needed in the rest of the paper.
More details concerning locally convex vector spaces and Lie groupoids can be found for example in \cite{Hor67,KrMi97,Tre67}
respectively \cite{Mac05,MoMr03,MoMr05}.

We will assume all our manifolds to be smooth, Hausdorff and paracompact, but not necessarily second-countable. For any
such manifold $M$ we will denote by $\C(M)$ the vector space of smooth $\CC$-valued functions on $M$. The subspaces
of compactly supported and $\RR$-valued functions on $M$ will be denoted by $\Cc(M)$ respectively $\C(M,\RR)$.

\subsection{Locally convex spaces}\label{Subsection Locally convex spaces}

All our locally convex vector spaces will be complex and Hausdorff. 
A subset $B$ of a locally convex space $E$ is bounded if and only if the set $p(B)$ is a bounded subset of $\RR$ for any continuous
seminorm $p$ on $E$. For locally convex vector spaces $E$ and $F$ we will denote by $\Hom(E,F)$ the space 
of all continuous linear maps from $E$ to $F$, equipped with the strong topology of uniform
convergence on bounded subsets. The basis of neighbourhoods of zero 
in $\Hom(E,F)$ consists of sets of the form
\[
K(B,V)=\{T\in\Hom(E,F)\,|\,T(B)\subset V\},
\]
where $B$ is a bounded subset of $E$ and $V$ is a neighbourhood of zero in $F$. 
If $E$ and $F$ are modules over an $\CC$-algebra $A$, we will denote by $\Hom_{A}(E,F)$ 
the corresponding space of continuous $A$-module homomorphisms and equip it with the induced topology from $\Hom(E,F)$. 

The space $\C(\RR^{l})$ has a structure of a Fr\'{e}chet algebra for any $l\in\NN$. Topology on $\C(\RR^{l})$
is generated by a family of seminorms $\{p_{L,m}\}$, indexed by compact subsets $L$ of $\RR^{l}$ and $m\in\NN_{0}$, given by
\[
p_{L,m}(F)=\sup_{\substack{x\in L,|\alpha|\leq m}}|D^{\alpha}(F)(x)|
\]
for $F\in\C(\RR^{l})$. Here we denoted $D^{\alpha}(F)=\frac{\partial^{|\alpha|}F}{\partial x_{1}^{\alpha_{1}}\cdots\partial x_{l}^{\alpha_{l}}}$, 
where $\alpha=(\alpha_{1},\ldots,\alpha_{l})\in\NN_{0}^{l}$ is a multi-index and $|\alpha|=\alpha_{1}+\alpha_{2}+\ldots+\alpha_{l}$. 
If $M$ is a second-countable manifold, one can choose similar seminorms with respect to some open cover of $M$ with
local coordinate charts to define the Fr\'{e}chet topology on $\C(M)$. This topology 
coincides with the topology of uniform convergence of all derivatives on compact subsets of $M$.
The strong dual of the space $\C(M)$ is the space $\cE'(M)=\Hom(\C(M),\CC)$ of compactly supported distributions on $M$.

If $M$ is not compact, the subspace $\Cc(M)$ of $\C(M)$ is not complete in the Fr\'{e}chet topology, 
so we consider a finer LF-topology on $\Cc(M)$. For any compact subset $L$ of $M$ we denote by $\Cc(L)$ the subspace of functions with
support contained in $L$. The space $\Cc(L)$ is a closed subspace of $\C(M)$ and hence a Fr\'{e}chet space itself.
The LF-topology on $\Cc(M)$ is now defined as the inductive limit topology with respect to the family of all subspaces
of the form $\Cc(L)$ for $L\subset M$ compact. The space 
$\Cc(M)$ with LF-topology is a complete locally convex space, which is not metrizable, if $M$ is not compact.

If $M$ is a smooth manifold and $E$ is a locally convex vector space,
a vector valued function $u:M\to E$ is smooth if in local coordinates all partial derivatives exist
and are continuous. 
We will denote by $\C(M,E)$ the space of smooth functions on $M$ with values in $E$ and by
$\Cc(M,E)$ its subspace, consisting of compactly supported functions. To make a distinction between scalar functions and vector valued functions, 
we will denote by $f(x)\in\CC$ the value of a function $f\in\C(M)$ at $x$ and by $u_{x}\in E$ the value of a function $u\in\C(M,E)$ at $x$.

\subsection{Lie groupoids}

A {\em Lie groupoid} is given by a manifold $M$ of objects and a manifold $\GG$ of arrows together with smooth structure maps:
target $\trg:\GG\to M$, source $\src:\GG\to M$, multiplication $\mlt:\GG\times^{\src,\trg}_{M}\GG\to\GG$, inverse $\inv:\GG\to\GG$
and unit $\uni:M\to\GG$. We assume that the source and the target maps are submersions to ensure that $\GG\times^{\src,\trg}_{M}\GG$
is a smooth manifold. A Lie groupoid is {\em \'{e}tale} if all its structure maps are local diffeomorphisms. Note that there exist
more general definitions of Lie groupoids, which we will not need.

\begin{ex}\rm
We will be mostly interested in action Lie groupoids. Suppose $H$ is a Lie group which acts
from the right on the manifold $M$. The associated action Lie groupoid $\GG=M\rtimes H$ is then a Lie groupoid over $M$ with
the manifold of arrows $M\times H$ and with the following structure maps:
\begin{align*}
\trg(x,h)&=x, \\
\src(x,h)&=xh, \\
\mlt((x,h),(xh,h'))&=(x,h)(xh,h')=(x,hh'), \\
\inv(x,h)&=(x,h)^{-1}=(xh,h^{-1}), \\
\uni(x)&=(x,e).
\end{align*}
Here $x\in M$ and $h,h'\in H$ are arbitrary, while $e$ is the unit of the Lie group $H$. The action groupoid $\GG$ is
\'{e}tale if and only if the group $H$ is discrete.
\end{ex}

\subsection{Real commutative algebras}

Let $A$ be an $\RR$-algebra. A real character on $A$ is 
a nontrivial multiplicative homomorphism from $A$ to $\RR$. We will denote by
\[
\Spec(A)
\]
the space of all real characters on $A$, equipped with the Gelfand topology (i.e. the relative weak-$^{*}$ topology). 
If the algebra $A$ satisfies the conditions of the Theorem in \cite{MiVa96}, the space $\Spec(A)$ also has
a natural smooth structure. 

If $Q$ is a smooth manifold, we have 
\[
\Spec(\C(Q))=\{\delta_{q}\,|\,q\in Q\},
\]
where $\delta_{q}$ is the Dirac functional, concentrated at the point $q$.
In this case we can equip the set $\Spec(\C(Q))$ with a topology and a smooth structure
such that the map $\Phi^{\text{man}}_{Q}:Q\to\Spec(\C(Q))$, defined by $\Phi^{\text{man}}_{Q}(q)=\delta_{q}$, is a diffeomorphism.

\subsection{Coalgebras}\label{Subsection Coalgebras}

Let $R$ be a commutative ring. We say that $R$ has local identities if for
any $r_{1},\ldots,r_{n}\in R$ there exists $r\in R$ such that $rr_{i}=r_{i}$ for $i=1,\ldots,n$. Similarly, a left $R$-module $C$
is locally unitary if for any $c_{1},\ldots,c_{n}\in C$ there exists $r\in R$ such that $rc_{i}=c_{i}$ for $i=1,\ldots,n$.

Suppose now that $R$ is an associative, commutative algebra with local identities over the field of complex numbers $\CC$ 
and let $C$ be a locally unitary left $R$-module. A coalgebra structure on $C$ over $R$ consists of two $R$-linear maps:
\begin{align*}
\Delta&:C\to C\otimes_{R}C, \\
\epsilon&:C\to R,
\end{align*}
called comultiplication and counit, which satisfy the conditions $(\id\ten\epsilon)\com\Delta=\id$ and 
$(\epsilon\ten\id)\com\Delta=\id$. Note that these conditions make sense because we can identify $C$ with $R\ten_{R}C$ and
$C\ten_{R}R$ since $C$ is locally unitary. A {\em coalgebra} over $R$ is a locally unitary left $R$-module $C$, equipped
with a coalgebra structure $(\Delta,\epsilon)$, which is coassociative in the sense that 
$(\id\ten\Delta)\com\Delta=(\Delta\ten\id)\com\Delta$. A coalgebra $C$ over $R$ is cocommutative if $\sigma\com\Delta=\Delta$,
where the flip isomorphism $\sigma:C\ten_{R}C\to C\ten_{R}C$ is given by $\sigma(c\ten c')=c'\ten c$.

Our main examples of coalgebras will be coalgebras associated to sheaves, which were introduced in \cite{Mrc07_2}.

\begin{ex}\rm\label{Example sheaf coalgebras}
Let $P$ and $M$ be manifolds and let
$\pi:P\to M$ be a local diffeomorphism (i.e. $P$ is a sheaf over $M$). The ring $\Cc(M)$ always has local units, but it is unital if and only 
if $M$ is compact. It will be convenient to denote for any $f\in\Cc(M)$ by $1_{f}\in\Cc(M)$ a function which
satisfies $1_{f}f=f$. If $\supp(f)\neq M$, the function $1_{f}$ is not uniquely defined.

The space $\Cc(P)$ has a natural structure of a locally unitary left $\Cc(M)$-module with the multiplication given by
\[
(f\cdot F)(p)=f(\pi(p))F(p)
\]
for $f\in\Cc(M)$, $F\in\Cc(P)$ and $p\in P$. The counit $\epsilon:\Cc(P)\to\Cc(M)$ is given by
\[
\epsilon(F)(x)=\sum_{p\in\pi^{-1}(x)}F(p)
\]
for any $F\in\Cc(P)$ and any $x\in M$. Note that the above sum is finite since $F$ has compact support and the fibers of $\pi$
are discrete.

To describe the comultiplication, let us recall that an open subset $W\subset P$ is $\pi$-elementary if $\pi|_{W}:W\to\pi(W)$ 
is a diffeomorphism. For any $f\in\Cc(\pi(W))$ we can consider the function $f\com\pi|_{W}\in\Cc(W)$ as an element of $\Cc(P)$,
if we extend it by zero outside of $W$. Since any $a\in\Cc(P)$ has compact support, it can be written (nonuniquely) as a finite sum
$a=\sum_{i=1}^{n}f_{i}\com\pi|_{W_{i}}$, where $W_{1},\ldots,W_{n}$ are $\pi$-elementary subsets of $P$ and $f_{i}\in\Cc(\pi(W_{i}))$
for $i=1,\ldots,n$. The comultiplication $\Delta:\Cc(P)\to\Cc(P)\ten_{\Cc(M)}\Cc(P)$ can be now defined by
\[
\Delta(\sum_{i=1}^{n}f_{i}\com\pi|_{W_{i}})=\sum_{i=1}^{n}(f_{i}\com\pi|_{W_{i}})\ten(1_{f_{i}}\com\pi|_{W_{i}}).
\]
One can show that the definition of $\Delta$ is independent of the various choices that we have made and that we obtain in this
way a cocommutative coalgebra $\Cc(P)$ over $\Cc(M)$.

For our purposes we will be mostly interested in trivial sheaves. If $\Gamma$ is a discrete topological space, then 
the projection $\pi:M\times\Gamma\to M$ is a local diffeomorphism. We can decompose the vector space $\Cc(M\times\Gamma)$ 
as a direct sum
\[
\Cc(M\times\Gamma)=\bigoplus_{y\in\Gamma}\Cc(M\times\{y\}).
\]
Using this decomposition we can write every element $a\in\Cc(M\times\Gamma)$ uniquely in the form
\[
a=\sum_{i=1}^{n}f_{i}\cdot\delta_{y_{i}}
\]
for some $f_{1},\ldots,f_{n}\in\Cc(M)$ and some $y_{1},\ldots,y_{n}\in\Gamma$. Here we have denoted 
for any $f\in\Cc(M)$ and any $y\in\Gamma$ by $f\cdot\delta_{y}\in\Cc(M\times\Gamma)$ the function, given by
\[
(f\cdot\delta_{y})(x,y')=
\left\{ \begin{array} {ll}
      f(x)   &;\hspace{1mm}y'=y,\\
      0      &;\hspace{1mm}y'\neq y
       \end{array}\right.
\]
for $(x,y')\in M\times\Gamma$. The comultiplication and counit are then given on the generators of 
$\Cc(M\times\Gamma)$ by the formulas:
\begin{align*}
\cm(f\cdot\delta_{y})&=(f\cdot\delta_{y})\ten(1_{f}\cdot\delta_{y}), \\
\cu(f\cdot\delta_{y})&=f.
\end{align*}
\end{ex}

In the rest of this subsection we will focus on coalgebras over the algebra $\Cc(M)$ for some manifold $M$ and recall the 
main results from \cite{Mrc07_2}. For any $x\in M$ we denote by
\[
I_{x}=\{f\in\Cc(M)\,|\,\text{$f|_{U}=0$ for some neighbourhood $U$ of $x$}\}
\]
the ideal of $\Cc(M)$, consisting of all functions with trivial germ at $x$. The quotient algebra of $\Cc(M)$
with respect to this ideal will be denoted by
\[
\Cc(M)_{x}=\Cc(M)/I_{x}.
\]
Now let $C$ be a coalgebra over $\Cc(M)$. The quotient
\[
C_{x}=C/I_{x}C
\]
then inherits a structure $(\Delta_{x},\epsilon_{x})$ of a coalgebra over $\Cc(M)_{x}$ 
which is called the local coalgebra of $C$ at $x$. The image of $c\in C$ in the quotient $C_{x}$ will be denoted by $c|_{x}\in C_{x}$.

An element $c\in C$ is weakly grouplike if $\Delta(c)=c\ten c'$ for some $c'\in C$. A weakly grouplike element $c\in C$ is 
normalized on an open subset $U\subset M$ if $\epsilon_{x}(c|_{x})=1$ for all $x\in U$. Weakly grouplike elements of the sheaf coalgebra 
$\Cc(P)$ are precisely elements of the form $f\com\pi|_{W}$ for some $\pi$-elementary
open subset $W$ of $P$ and some $f\in\Cc(\pi(W))$. Such an element $f\com\pi|_{W}$ is normalized on $U\subset\pi(W)$ if $f|_{U}\equiv 1$.

Since the local ring $\Cc(M)_{x}$ is unital for any $x\in M$,
we can also define the set of grouplike elements of $C_{x}$ by
\[
G(C_{x})=\{\xi\in C_{x}\,|\,\Delta_{x}(\xi)=\xi\ten\xi,\,\epsilon_{x}(\xi)=1\}.
\]
An element $\xi\in C_{x}$ is grouplike if and only if $\xi=c|_{x}$ for some weakly grouplike element $c\in C$, which is normalized on some
open neighbourhood of $x$. 

The spectral sheaf $\Esp(C)$ of a $\Cc(M)$-coalgebra $C$ is the sheaf 
\[
\pisp(C):\Esp(C)\to M
\] 
with the stalk
\[
\Esp(C)_{x}=G(C_{x}).
\]
The topology on $\Esp(C)$ is defined by the basis, consisting of $\pisp(C)$-elementary subsets of $\Esp(C)$ of the form
\[
c|_{U}=\{c|_{x}\in\Esp(C)\,|\,x\in U\}\subset\Esp(C),
\]
where $c\in C$ is a weakly grouplike element, normalized on an open subset $U$ of $M$.

Now let $\pi:P\to M$ be a sheaf over $M$. By Theorem $2.4$ in \cite{Mrc07_2} we have a natural isomorphism of sheaves
\[
\Phi^{\text{shv}}_{P}:P\to\Esp(\Cc(P))
\]
defined by
\[
\Phi^{\text{shv}}_{P}(p)=(f\com\pi|_{W})|_{\pi(p)},
\]
where $p\in P$, $W$ is a $\pi$-elementary neighbourhood of $p$ in $P$ and $f\in\Cc(\pi(W))$ is such that $f|_{\pi(p)}=1\in\Cc(M)_{\pi(p)}$. 
Moreover, by Theorem $2.10$ in \cite{Mrc07_2}, a coalgebra $C$ is isomorphic to some sheaf coalgebra $\Cc(P)$ if and only
if $C$ is locally grouplike, which by definition means that for every $x\in M$ the $\Cc(M)_{x}$-module $C_{x}$ is
free with the basis $G(C_{x})$.

\subsection{Bialgebroids and Hopf algebroids}

Bialgebroids and Hopf algebroids are generalizations of bialgebras and Hopf algebras over arbitrary rings.
In the literature \cite{Boh09,KaSa19,KoPo09,Lu96,Mal00,Swe74,Tak77,Xu98} one can find several similar definitions, which are in general inequivalent.
Our definition follows the one in \cite{Mrc07_1}.

Let $R$ be a commutative $\CC$-algebra with local units.
We say that a $\CC$-algebra $A$ extends $R$ if $
R$ is a subalgebra of $A$ and $A$ has local units in $R$. We do not assume that $R$ is a central subalgebra of $A$.
Any $\CC$-algebra $A$ which extends $R$, is naturally
an $R$-$R$-bimodule. We will denote by $A\ten_R A=A\ten_R^{ll} A$ the 
tensor product of left $R$-modules, which has two natural
right $R$-module structures. 

A {\em bialgebroid} over $R$ is a $\CC$-algebra $A$ which extends $R$,
together with structure maps $\Delta:A\to A\ten_{R}A$ and $\epsilon:A\to R$
for which $(A,\Delta,\epsilon)$ is a cocommutative coalgebra and such that:
\begin{enumerate}
\item [(i)]   $\cm(A)\subset A\overline{\ten}_{R}A$, where
              $A\overline{\ten}_{R}A$ is the algebra consisting of those
              elements of $A\ten_{R}A$, on which both right $R$-actions coincide,
\item [(ii)]  $\cu|_{R}=\id$ and $\cm|_{R}$ is the
              canonical embedding $R\subset A\ten_{R}A$,
\item [(iii)] $\cu(ab)=\cu(a\cu(b))$ and $\cm(ab)=\cm(a)\cm(b)$ for any $a,b\in A$.
\end{enumerate}
Antipode on a bialgebroid $A$ is a $\CC$-linear involution
$S:A\to A$ which satisfies the conditions:
\begin{enumerate}
\item [(i)] $S|_{R}=\id$ and $S(ab)=S(b)S(a)$ for any $a,b\in A$,
\item [(ii)]  $\mu_{A}\com(S\ten\id)\com\cm=\cu\com S$, where
$\mu_{A}$ denotes the multiplication in $A$.
\end{enumerate}
A {\em Hopf algebroid} over $R$ is a bialgebroid $A$ over $R$ with an antipode. Note that in the case 
when $R$ is a central subalgebra of $A$ the notions of bialgebroid and Hopf algebroid
coincide with the more familiar notions of bialgebra and Hopf algebra.

In the next example we will recall from \cite{Mrc07_1} the Hopf algebroid associated to an \'{e}tale Lie groupoid.

\begin{ex}\rm\label{Example Locally grouplike Hopf algebroid}
Let $\GG$ be an \'{e}tale Lie groupoid over $M$. Multiplication on $\GG$ induces
a convolution product \cite{Con82} on $\Cc(\GG)$, given by the formula
\[
(a_{1}*a_{2})(g)=\sum_{g=g_{1}g_{2}}a_{1}(g_{1})a_{2}(g_{2})
\]
for $a_{1},a_{2}\in\Cc(\GG)$. Note that this sum is always finite as $a_{1}$ and $a_{2}$
are compactly supported. Since $\trg:\GG\to M$ is a sheaf, we can also consider the space
$\Cc(\GG)$ as a locally grouplike coalgebra over $\Cc(M)$. Finally, the antipode 
$S:\Cc(\GG)\to\Cc(\GG)$ is defined by the formula
\[
S(a)=a\com\inv
\]
for $a\in\Cc(\GG)$. In this way $\Cc(\GG)$ becomes a Hopf algebroid over $\Cc(M)$.

Suppose now that $\Gamma$ is a discrete group which acts from the right on the manifold $M$
and denote by $M\rtimes\Gamma$ the associated action groupoid. We then have the decomposition
\[
\Cc(M\rtimes\Gamma)=\bigoplus_{g\in\Gamma}\Cc(M\times\{g\}).
\]
For any $g\in\Gamma$ and $f\in\Cc(M)$ let us denote by $gf\in\Cc(M)$ the function, given by $(gf)(x)=f(xg)$ for $x\in M$.
If we use the notation from the Example \ref{Example sheaf coalgebras},
the convolution product and the antipode on $\Cc(M\rtimes\Gamma)$ can be described on the set of generators by the formulas:
\begin{align*}
(f_{1}\cdot\delta_{g_{1}})*(f_{2}\cdot\delta_{g_{2}})&=(f_{1}(g_{1}f_{2}))\cdot\delta_{g_{1}g_{2}}, \\
S(f\cdot\delta_{g})&=(g^{-1}f)\cdot\delta_{g^{-1}}.
\end{align*}
\end{ex}

Now let $A$ be a Hopf algebroid over $\Cc(M)$. We will next recall from \cite{Mrc07_1} 
the construction of the spectral \'{e}tale Lie groupoid $\Gsp(A)$, 
associated to $A$. Note that $A$ is a coalgebra over $\Cc(M)$, so we have the notion of weakly grouplike elements. We say
that a weakly grouplike element $a\in A$ is $S$-invariant if there exists $a'\in A$ such that $\cm(a)=a\ten a'$ and
$\cm(S(a))=S(a')\ten S(a)$. In the case of the Hopf algebroid $\Cc(\GG)$ of an \'{e}tale Lie groupoid $\GG$, an element
$a\in\Cc(\GG)$ is $S$-invariant weakly grouplike if and only if it is of the form $f\com\trg|_{W}$ for some bisection $W$ of $\GG$
and some $f\in\Cc(\trg(W))$ (a bisection of an \'{e}tale Lie groupoid $\GG$ is an open subset $W$ of $\GG$ 
which is both $\trg$-elementary and $\src$-elementary).

An arrow of $A$ with target $y\in M$ is an element $g\in G(A_{y})$, for which there exists an
$S$-invariant weakly grouplike element $a\in A$ such that $g=a|_{y}$. The set of all arrows of $A$ with target $y$ will
be denoted by $\Gsp(A)_{y}$. All arrows of $A$ form a subsheaf $\Gsp(A)$ of $\Esp(A)$, whose projection will be denoted by
\[
\trg=\pisp(A)|_{\Gsp(A)}:\Gsp(A)\to M.
\]
To describe the source map of $\Gsp(A)$, we first recall that each $S$-invariant weakly grouplike element $a\in A$ induces 
a $\CC$-linear map $T_{a}:\Cc(M)\to\Cc(M)$, given by $T_{a}(f)=\cu(S(fa))$.
If $a$ is normalized on an open subset $U$ of $M$, one can find an open subset $U'$ of $M$ and a unique
diffeomorphism $\tau_{U,a}:U'\to U$ such that $T_{a}(\Cc(U))\subset\Cc(U')$ and $T_{a}(f)=f\com\tau_{U,a}$ 
for any $f\in\Cc(U)$. Furthermore, the element $S(a)$ is $S$-invariant weakly grouplike, normalized on $U'$ and we have
that $\tau_{U',S(a)}=\tau_{U,a}^{-1}$. The source map $\src:\Gsp(A)\to M$ is now defined by
\[
\src(a|_{y})=\tau_{U,a}^{-1}(y),
\]
where $a\in A$ is an $S$-invariant weakly grouplike element, normalized on $U$. Now choose elements $a,b\in A$ which
represent an arrow $a|_{y}\in\Gsp(A)_{y}$ and an arrow $b|_{x}\in\Gsp(A)_{x}$ such that $\src(a|_{y})=x$. The product
of $a|_{y}$ and $b|_{x}$ is then defined by
\[
a|_{y}b|_{x}=(ab)|_{y}.
\]
The unit $\uni(x)\in\Gsp(A)$ at the point $x\in M$ is given by
\[
\uni(x)=f|_{x},
\]
where $f\in\Cc(M)\subset A$ is any function with germ $f|_{x}=1\in\Cc(M)_{x}$. Finally, the inverse of an
arrow $a|_{y}\in\Gsp(A)_{y}$ with $\src(a|_{y})=x$ is defined by
\[
a|_{y}^{-1}=S(a)|_{x}.
\]
The groupoid $\Gsp(A)$ is an \'{e}tale Lie groupoid over $M$, called the spectral \'{e}tale Lie groupoid
of the Hopf algebroid $A$. For any \'{e}tale Lie groupoid $\GG$ over $M$
we have a natural isomorphism of Lie groupoids
\[
\Phi^{\text{egr}}_{\GG}:\GG\to\Gsp(\Cc(\GG)),
\]
defined by
\[
\Phi^{\text{egr}}_{\GG}(g)=(f\com\trg|_{W})|_{t(g)},
\]
where $W$ is any bisection of $\GG$ which contains $g$ and $f\in\Cc(\trg(W))$ is any function with $f|_{\trg(g)}=1\in\Cc(M)_{\trg(g)}$.

Hopf algebroid $A$ is isomorphic to a Hopf algebroid of the form $\Cc(\GG)$ for an \'{e}tale 
Lie groupoid $\GG$ if and only if it is locally grouplike, which means that for every $y\in M$ the
$\Cc(M)_{y}$-coalgebra $A_{y}$ is a free $\Cc(M)_{y}$-module with the basis consisting of arrows of $A$ at the point $y$.

\section{Transversal distributions of constant Dirac type}\label{Dirac distributions of constant type}

Let $\pi:M\times N\to M$ be a trivial bundle over $M$ with fiber $N$. In the spirit of the Gelfand-Naimark theorem
we will assign to it a locally convex $\Cc(M)$-module $\Dirac_{\pi}(M\times N)$ of distributions of constant Dirac type on $M\times N$
and show that its strong $\Cc(M)$-dual is isomorphic to the space $\C(M\times N)$. 
In general, the space $\Dirac_{\pi}(M\times N)$ is a dense subspace of the space
$\cE'_{\pi}(M\times N)$ of compactly supported transversal distributions. However, in the case when
$N=\Gamma$ is discrete, the space $\Dirac_{\pi}(M\times\Gamma)$ 
is complete and isomorphic to the LF-space $\Cc(M\times\Gamma)$. 

We start with the definition of transversal distributions on a trivial bundle.

\begin{dfn}
Let $M$ and $N$ be second-countable manifolds and let $M\times N$ be the trivial bundle over
$M$ with fibre $N$ and projection $\pi:M\times N\to M$.
The space of \textbf{$\pi$-transversal distributions with compact support} is the space 
\[
\cE'_{\pi}(M\times N)=\Hom_{\Cc(M)}(\C(M\times N),\Cc(M)).
\]
\end{dfn}

In other words, $\cE'_{\pi}(M\times N)$ is the space of continuous $\Cc(M)$-linear
maps from $\C(M\times N)$ to $\Cc(M)$, where the $\Cc(M)$-module structure on $\C(M\times N)$ is given by
\[
(f\cdot F)(x,y)=f(x)F(x,y)
\]
for $f\in\Cc(M)$, $F\in\C(M\times N)$ and $(x,y)\in M\times N$. The space $\cE'_{\pi}(M\times N)$ is a $\Cc(M)$-module as
well, with module structure given by
\[
(f\cdot T)(F)=T(f\cdot F)
\]
for $f\in\Cc(M)$, $F\in\C(M\times N)$ and $T\in\cE'_{\pi}(M\times N)$. If we equip the space $\cE'_{\pi}(M\times N)$
with the strong topology of uniform convergence on bounded subsets, it becomes a complete locally convex space.

\begin{rem}\rm

$(1)$ It was shown in \cite{LeMaVa17} that there is an isomorphism
\[
\cE'_{\pi}(M\times N)\cong\Cc(M,\cE'(N)),
\]
which enables us to identify a $\pi$-transversal distribution $T\in\cE'_{\pi}(M\times N)$ with a smooth, compactly supported
family $u=u(T)\in\Cc(M,\cE'(N))$. We will denote the value of $u$ at $x\in M$ by $u_{x}\in\cE'(N)$. If we denote by $\pi_{N}:M\times N\to N$
the projection to $N$, the distribution $u_{x}$ is given by the formula
\[
u_{x}(F)=T(F\com\pi_{N})(x)
\]
for any $F\in\C(N)$. We can also view any $u\in\Cc(M,\cE'(N))$ as a $\pi$-transversal distribution $T=T(u)$, if we define
\[
T(F)(x)=u_{x}(F\com\iota_{x})
\]
for $F\in\C(M\times N)$. Here $\iota_{x}:N\to\{x\}\times N$ is given by $\iota_{x}(y)=(x,y)$ for $x\in M$.
Different letters $T$ and $u$ are used intentionally to make a slight distinction between  
transversal distributions and smooth families of distributions along the fibres.

$(2)$ We can define a support of a family $u\in\Cc(M,\cE'(N))$ either as a subspace of $M$ or a subspace of $M\times N$. 
The support of $u$ is the subset $\supp(u)$ of $M$, defined by $\supp(u)=\overline{\{x\in M\,|\,u_{x}\neq 0\}}$. On the other hand,
the total support of $u$ is the subset $\supp_{M\times N}(u)$ of $M\times N$, consisting of all points 
$(x,y)\in M\times N$, which satisfy the condition that for every open neighbourhood $U$ of $(x,y)$ there exists $F\in\Cc(U)$ such that $u(F)\neq 0$.
For $u\in\Cc(M,\cE'(N))$ both supports are compact and we have $\pi(\supp_{M\times N}(u))=\supp(u)$. We can also define
the support of a $\pi$-transversal distribution $T\in\cE'_{\pi}(M\times N)$ as the subset $\supp(T)=\supp_{M\times N}(u(T))$ of $M\times N$.
For more details about supports we refer the reader to \cite{Kal22_2}.

$(3)$ If $\dim(M)>0$, the space $\Hom_{\Cc(M)}(\C(M\times N),\Cc(M))$ in fact coincides with the space $\Lin_{\Cc(M)}(\C(M\times N),\Cc(M))$ 
of $\Cc(M)$-linear maps, without any assumption on continuity (see \cite{Kal22}).
\end{rem}

Let us take a look at some important examples of transversal distributions that will be used throughout the paper.

\begin{ex}\label{Examples of transversal distributions} \rm

(1) Let $\pi:M\times N\to M$ be a trivial bundle and denote for any $y\in N$ by $E_{y}=M\times\{y\}$ the horizontal subspace of $M\times N$.
For any $f\in\Cc(M)$ we define a $\pi$-transversal distribution 
$\llbracket E_{y},f\rrbracket\in\cE'_{\pi}(M\times N)$ by
\[
\llbracket E_{y},f\rrbracket(F)(x)=f(x)F(x,y),
\]
for $F\in\C(M\times N)$. We think of $\llbracket E_{y},f\rrbracket$ as a smooth family of Dirac distributions, 
supported on the constant section $E_{y}$ and weighted by
the function $f$. In particular, we have
\[
\llbracket E_{y},f\rrbracket_{x}=f(x)\delta_{y}.
\] 

(2) Let $M=\RR^{l},\,N=\RR^{k}$ and let $\pi:\RR^{l}\times\RR^{k}\to\RR^{l}$ be the projection onto $\RR^{l}$. For  
$\phi\in\Cc(\RR^{l}\times\RR^{k})$ we define a $\pi$-transversal distribution $T_{\phi}\in\cE'_{\pi}(\RR^{l}\times\RR^{k})$ by
\[
T_{\phi}(F)(x)=\int_{\RR^{k}}\phi(x,y)F(x,y)\,dy
\]
for $F\in\C(\RR^{l}\times\RR^{k})$. The distribution $T_{\phi}$ corresponds to the family of smooth densities on $\RR^{k}$, parametrized by $\RR^{l}$
and explicitly given by
\[
(T_{\phi})_{x}=\phi(x,-)dV,
\]
where $dV$ is the Lebesgue measure on $\RR^{k}$. 

The map $\phi\mapsto T_{\phi}$ defines a continuous, injective $\Cc(\RR^{l})$-linear map 
\[
\Cc(\RR^{l}\times\RR^{k})\hookrightarrow\cE'_{\pi}(\RR^{l}\times\RR^{k}).
\]
Note that the $LF$-topology on $\Cc(\RR^{l}\times\RR^{k})$ is strictly finer than the subspace topology that is induced from  
$\cE'_{\pi}(\RR^{l}\times\RR^{k})$ via the above map.
In particular, if $M=\RR^{0}$ is a point, the above construction enables us to consider the space $\Cc(\RR^{k})$ as a subspace
of the space $\cE'(\RR^{k})$.
\end{ex}

Let us now denote for a manifold $N$ by $\Nd$ the set $N$ with the discrete topology.
The projection $\pid:M\times\Nd\to M$ is then a local diffeomorphism. Note that the manifold $M\times\Nd$ is paracompact, 
but not second-countable if $\dim(N)>0$. 

Using the notation from the Example \ref{Example sheaf coalgebras} we have a decomposition
\[
\Cc(M\times\Nd)=\bigoplus_{y\in\Nd}\Cc(M\times\{y\}),
\]
which enables us to write every element $a\in\Cc(M\times\Nd)$ uniquely in the form
\[
a=\sum_{i=1}^{n}f_{i}\cdot\delta_{y_{i}}
\]
for some $f_{1},\ldots,f_{n}\in\Cc(M)$ and some $y_{1},\ldots,y_{n}\in N$. 
Now define an injective $\Cc(M)$-linear map $\Psi_{M\times N}:\Cc(M\times\Nd)\to\cE'_{\pi}(M\times N)$ by
\[
\Psi_{M\times N}\big(\sum_{i=1}^{n}f_{i}\cdot\delta_{y_{i}}\big)=\sum_{i=1}^{n}\llbracket E_{y_{i}},f_{i}\rrbracket.
\]

\begin{dfn}
Let $M\times N$ be a trivial bundle with projection $\pi:M\times N\to M$. The space of \textbf{$\pi$-transversal distributions of
constant Dirac type} is the space
\[
\Dirac_{\pi}(M\times N)=\Psi_{M\times N}(\Cc(M\times\Nd))\subset\cE'_{\pi}(M\times N).
\]
The space $\Dirac_{\pi}(M\times N)$ is equipped with the induced topology from $\cE'_{\pi}(M\times N)$.

If $M$ is a point, we denote by $\Dirac(N)$ the subspace of $\cE'(N)$, spanned by Dirac distributions.
\end{dfn}

We will show in the sequel that $\Dirac_{\pi}(M\times N)$ is a proper, dense subspace of $\cE'_{\pi}(M\times N)$
if $\dim(N)>0$. However, in the case when $N=\Gamma$ is discrete, we have 
the following description of the space $\Dirac_{\pi}(M\times\Gamma)$.

\begin{prop}
Let $M\times\Gamma$ be a trivial bundle over a second-countable manifold $M$ with a 
countable discrete fiber $\Gamma$ and bundle projection $\pi:M\times\Gamma\to M$. 
\begin{enumerate}
\item[(a)] The map $\Psi_{M\times\Gamma}:\Cc(M\times\Gamma)\to\cE'_{\pi}(M\times\Gamma)$ is an isomorphism of $\Cc(M)$-modules, 
so we have $\Dirac_{\pi}(M\times\Gamma)=\cE'_{\pi}(M\times\Gamma)$. 
\item[(b)] The map $\Psi_{M\times\Gamma}$ is an isomorphism of locally convex spaces with respect
to the $LF$-topology on $\Cc(M\times\Gamma)$ and the strong topology on $\cE'_{\pi}(M\times\Gamma)$.
\end{enumerate}
\end{prop}
\begin{proof} 
(a) First recall that we have an isomorphism $\cE'_{\pi}(M\times\Gamma)\cong\Cc(M,\cE'(\Gamma))$ of $\Cc(M)$-modules.
It is therefore enough to show that every $u\in\Cc(M,\cE'(\Gamma))$ is of the form $u=\Psi_{M\times\Gamma}(a)$ for some 
$a\in\Cc(M\times\Gamma)$.

Since $\Gamma$ is discrete, the space $\cE'(\Gamma)$ is isomorphic to the locally convex direct sum 
$\bigoplus_{y\in\Gamma}\text{Span}(\delta_{y})$ of one-dimensional subspaces, spanned by Dirac distributions. 
Any $u\in\Cc(M,\cE'(\Gamma))$ has compact support, so its image $u(M)\subset\cE'(\Gamma)$ is compact and hence bounded.
This implies that $u(M)$ lies in some finite-dimensional subspace $\bigoplus_{i=1}^{n}\text{Span}(\delta_{y_{i}})\subset\cE'(\Gamma)$
for some $y_{1},\ldots,y_{n}\in\Gamma$. We can therefore find functions $f_{1},\ldots,f_{n}:M\to\CC$ such that
\[
u_{x}=\sum_{i=1}^{n}f_{i}(x)\delta_{y_{i}}
\]
for every $x\in M$. If we denote by $1_{y_{i}}\in\C(M\times\Gamma)$ the function, which is equal to $1$ on $M\times\{y_{i}\}$ and zero elsewhere,
we have $f_{i}=u(1_{y_{i}})\in\Cc(M)$ and therefore
\[
u=\Psi_{M\times\Gamma}\big(\sum_{i=1}^{n}f_{i}\cdot\delta_{y_{i}}\big).
\]

(b) To see that $\Psi_{M\times\Gamma}$ is continuous, first choose a basic neighbourhood $K(B,V)$ of zero in $\cE'_{\pi}(M\times\Gamma)$, where $B$
is a bounded subset of $\C(M\times\Gamma)$ and $V$ is a neighbourhood of zero in $\Cc(M)$. From the definition of $LF$-topology on $\Cc(M\times\Gamma)$ 
it follows that we only have to show that the restrictions of $\Psi_{M\times\Gamma}$ onto subspaces of the form $\Cc(L\times\{y\})$ are continuous 
for all $y\in\Gamma$ and all compact subsets $L$ of $M$. Define a multiplication map $\mu:\Cc(L\times\{y\})\times\C(M\times\Gamma)\to\Cc(M)$ by
\[
\mu(f\cdot\delta_{y},F)(x)=f(x)F(x,y)
\]
for $f\cdot\delta_{y}\in\Cc(L\times\{y\})$ and $F\in\C(M\times\Gamma)$. Note that $\mu$ is continuous, 
so we can find neighbourhoods $V_{1}$ and $V_{2}$ of zero in
$\Cc(L\times\{y\})$ respectively $\C(M\times\Gamma)$ such that $\mu(V_{1},V_{2})\subset V$. Since $B\subset\C(M\times\Gamma)$ is bounded, we can 
assume that $B\subset V_{2}$ (if necessary, rescale $V_{1}$ and $V_{2}$ by appropriate inverse factors). Now observe that
\[
\mu(f\cdot\delta_{y},F)=\Psi_{M\times\Gamma}(f\cdot\delta_{y})(F).
\]
For $f\cdot\delta_{y}\in V_{1}$ and $F\in B$ we thus have $\Psi_{M\times\Gamma}(f\cdot\delta_{y})(F)=\mu(f\cdot\delta_{y},F)\in V$, 
which shows that $\Psi_{M\times\Gamma}(V_{1})\subset K(B,V)$.

Finally, we have to show that the map $\Psi_{M\times\Gamma}^{-1}$ is continuous. Let us choose a net $(u_{\alpha})_{\alpha\in A}$ that converges to zero in
$\cE'_{\pi}(M\times\Gamma)$. The set $\{u_{\alpha}\,|\,\alpha\in A\}$ is then a bounded subset of $\cE'_{\pi}(M\times\Gamma)$, so there
exists a compact subset of $M\times\Gamma$ which contains all supports $\supp(u_{\alpha})$ for $\alpha\in A$. In particular, we can
find a compact subset $L$ of $M$ and elements $y_{1},y_{2},\ldots,y_{n}\in\Gamma$, such that for every $\alpha\in A$ we can write
\[
u_{\alpha}=\sum_{i=1}^{n}\llbracket E_{y_{i}},f_{\alpha,i}\rrbracket
\]
for some $f_{\alpha,1},\ldots,f_{\alpha,n}\in\Cc(L)$. 
If we evaluate $u_{\alpha}$ at $1_{y_{i}}$ (see part (a) of the proof for the definition of $1_{y_{i}}$), we get that
$u_{\alpha}(1_{y_{i}})=f_{\alpha,i}$ converges to zero in $\Cc(M)$ for $1\leq i\leq n$, which implies that
$\Psi_{M\times\Gamma}^{-1}(u_{\alpha})=\sum_{i=1}^{n}f_{\alpha,i}\cdot\delta_{y_{i}}$ converges to zero in $\Cc(M\times\Gamma)$.
\end{proof}

We will now move on to the study of the space $\Dirac_{\pi}(M\times N)$ in the case of non-discrete fibre and show that it is a dense subspace of
the space $\cE'_{\pi}(M\times N)$. 

To do that, we first recall some facts about the convolution of distributions on euclidean spaces.
We will use the definition of the convolution product on $\cE'(\RR^{k})$ that is easily generalized to arbitrary Lie groupoids 
(see Section \ref{Section Locally convex bialgebroid}). 
For any $F\in\C(\RR^{k})$ we can define a smooth map $\RR^{k}\to\C(\RR^{k})$ by $y\to F\circ L_{y}$,
where the left translation $L_{y}:\RR^{k}\to\RR^{k}$ is defined by $L_{y}(y')=y+y'$ for $y\in\RR^{k}$. If we compose this map
with an arbitrary distribution $w\in\cE'(\RR^{k})$, we thus get a smooth map $\RR^{k}\to\CC$, given by $y\to w(F\circ L_{y})$.
The convolution product $*:\cE'(\RR^{k})\times\cE'(\RR^{k})\to\cE'(\RR^{k})$ can be then described by the formula
\[
(v*w)(F)=v(y\to w(F\circ L_{y}))
\]
for $F\in\C(\RR^{k})$ and $v,w\in\cE'(\RR^{k})$. The convolution product is a bilinear, jointly continuous map and it turns
$\cE'(\RR^{k})$ into a commutative, locally convex algebra.

As we have seen in the Example \ref{Examples of transversal distributions}, we can consider $\Cc(\RR^{k})$
as a subspace of $\cE'(\RR^{k})$. We can explicitly
describe the convolution of an arbitrary distribution with a smooth function as follows. Choose any $\rho\in\Cc(\RR^{k})$
and consider it as an element $T_{\rho}\in\cE'(\RR^{k})$, which we will for simplicity denote just by $\rho$. Let
$\tilde{\rho}\in\C(\RR^{k})$ be defined by $\tilde{\rho}(z)=\rho(-z)$. For any $v\in\cE'(\RR^{k})$ we then
have that $v*\rho\in\Cc(\RR^{k})\subset\cE'(\RR^{k})$ is a smooth function, given by
\[
(v*\rho)(y)=v(\tilde{\rho}\com L_{-y})
\]
for $y\in\RR^{k}$. This shows that $\Cc(\RR^{k})$ is actually an ideal of $\cE'(\RR^{k})$.

We will now use these results in the setting of transversal distributions. For any $u=(u_{x})_{x\in\RR^{l}}\in\Cc(\RR^{l},\cE'(\RR^{k}))$
and any $\rho\in\Cc(\RR^{k})$ we define $u*\rho\in\Cc(\RR^{l},\cE'(\RR^{k}))$ pointwise by
\[
(u*\rho)_{x}=u_{x}*\rho.
\]
Smoothness of $u*\rho$ follows from bilinearity and continuity of the convolution product. One can moreover show that
$u*\rho$ is actually of the form $u*\rho=T_{\phi}$ for the smooth function $\phi\in\Cc(\RR^{l}\times\RR^{k})$, given by
\[
\phi(x,y)=u_{x}(\tilde{\rho}\com L_{-y}).
\]

We will next show that the image of the map $\Cc(\RR^{l}\times\RR^{k})\hookrightarrow\Cc(\RR^{l},\cE'(\RR^{k}))$, given by $\phi\mapsto T_{\phi}$, 
is a dense subspace of $\Cc(\RR^{l},\cE'(\RR^{k}))$. 

To do that, we first recall from \cite{Kal22_2} an explicit description of a neighbourhood basis of zero in the space $\Cc(\RR^{l},\cE'(\RR^{k}))$.
Denote $K_{0}=\emptyset$ and let $K_{n}=\{x\in\RR^{l}\,|\,|x|\leq n\}$ be the ball
with centre at zero and radius $n\in\NN$. Choose an increasing sequence of natural numbers $\textbf{m}=(m_{1},m_{2},\ldots)$, 
a decreasing sequence of positive real numbers $\textbf{e}=(\epsilon_{1},\epsilon_{2},\ldots)$ and 
let $\textbf{B}=(B_{1},B_{2},\ldots)$ be an increasing sequence of bounded subsets of $\C(\RR^{k})$.
Now define a subset $V_{\textbf{B},\textbf{m},\textbf{e}}\subset\Cc(\RR^{l},\cE'(\RR^{k}))$ by
\[
V_{\textbf{B},\textbf{m},\textbf{e}}=\{u\in\Cc(\RR^{l},\cE'(\RR^{k}))\,|\,
p_{B_{n}}((D^{\alpha}u)_{x})<\epsilon_{n}\text{ for }x\in K_{n-1}^{c}\text{ and }|\alpha|\leq m_{n}\},
\]
where the seminorm $p_{B_{n}}$ on $\cE'(\RR^{k})$ is given by 
\[
p_{B_{n}}(v)=\sup_{F\in B_{n}}|v(F)|
\] 
for $v\in\cE'(\RR^{k})$. The family of all such sets $V_{\textbf{B},\textbf{m},\textbf{e}}$, with $\textbf{B}$, $\textbf{m}$ and $\textbf{e}$
as defined above, then forms a basis of neighbourhoods of zero for a topology on $\Cc(\RR^{l},\cE'(\RR^{k}))$, for which the natural identification
$\Cc(\RR^{l},\cE'(\RR^{k}))\cong\cE'_{\pi}(\RR^{l}\times\RR^{k})$ becomes an isomorphism of locally convex vector spaces (see \cite{Kal22_2}).

\begin{prop}\label{Smooth functions are dense in the space of distributions}
The image of the map $\Cc(\RR^{l}\times\RR^{k})\hookrightarrow\Cc(\RR^{l},\cE'(\RR^{k}))$, given by $\phi\mapsto T_{\phi}$, 
is a dense subspace of $\Cc(\RR^{l},\cE'(\RR^{k}))$.
\end{prop}
\begin{proof} 
Let us choose a one-parameter family $(\rho_{t})\in\Cc(\RR^{k})\subset\cE'(\RR^{k})$, for $t\in(0,1)$, which converges to the Dirac
distribution $\delta_{0}\in\cE'(\RR^{k})$ as $t\to 0$.

Now choose any $u\in\Cc(\RR^{l},\cE'(\RR^{k}))$ and define $u_{t}=u*\rho_{t}\in\Cc(\RR^{l}\times\RR^{k})$ for $t\in(0,1)$. We will show that
$u_{t}\to u$ as $t\to 0$. To do that, take an arbitrary basic neighbourhood of zero in $\Cc(\RR^{l},\cE'(\RR^{k}))$
of the form $V_{\textbf{B},\textbf{m},\textbf{e}}$. Since $u$ is compactly supported, we have $\supp(u)=\supp(u_{t})\subset K_{n}$
for some $n\in\NN$ and all $t\in(0,1)$. We need to show that for $t$ small enough, we have 
$u-u_{t}\in V_{\textbf{B},\textbf{m},\textbf{e}}$, which by the above observation means that $p_{B_{n}}(D^{\alpha}(u-u_{t})_{x})<\epsilon_{n}$ for
all $x\in K_{n}$ and all $\alpha$ with $|\alpha|\leq m_{n}$. Equivalently, if we denote $\DD(\epsilon_{n})=\{z\in\CC\,|\,|z|<\epsilon_{n}\}$,
then for all such $x$ and $\alpha$ we have
\[
D^{\alpha}(u-u_{t})_{x}\in K(B_{n},\DD(\epsilon_{n}))\subset\cE'(\RR^{k}).
\] 
Now note that the set $A=\{(D^{\alpha}u)_{x}\,|\,x\in K_{n},|\alpha|\leq m_{n}\}$ is compact and hence a bounded subset of 
$\cE'(\RR^{k})$. Since the convolution $*:\cE'(\RR^{k})\times\cE'(\RR^{k})\to\cE'(\RR^{k})$ is bilinear and continuous,
we can find a neighbourhood $V$ of zero in $\cE'(\RR^{k})$ such that $V*A\subset K(B_{n},\DD(\epsilon_{n}))$. Since 
$\rho_{t}\to\delta_{0}$ as $t\to 0$, we have that $\delta_{0}-\rho_{t}\in V$ for $t$ small enough and hence
\[
D^{\alpha}(u-u_{t})_{x}=(D^{\alpha}u)_{x}-(D^{\alpha}u*\rho_{t})_{x}
=(\delta_{0}-\rho_{t})*(D^{\alpha}u)_{x}\in K(B_{n},\DD(\epsilon_{n})).
\]
\end{proof}

We will next show, by using ideas from Riemannian integration, 
that arbitrary transversal distribution of the form $T_{\phi}\in\Cc(\RR^{l},\cE'(\RR^{k}))$
can be approximated by elements of $\Dirac_{\pi}(\RR^{l}\times\RR^{k})$.

Choose any $L>0$ and any $n\in\NN$ and denote $t_{j}=-\frac{L}{2}+\frac{jL}{n}$ for $0\leq j\leq n-1$.
The set $I=\{t_{0},t_{1},\ldots,t_{n-1}\}^{k}$ is then a finite subset of the cube $D=[-\frac{L}{2},\frac{L}{2}]^{k}$.
If we denote for $t\in I$ by $D_{t}=t+[0,\frac{L}{n}]^{k}$ the cube with volume $\text{vol}(D_{t})=\left(\frac{L}{n}\right)^{k}$,
we can express $D=\bigcup_{t\in I}D_{t}$ as a union of $n^{k}$ such small cubes. Now define a distribution 
$\Delta_{n}\in\Dirac(\RR^{k})\subset\cE'(\RR^{k})$ by
\[
\Delta_{n}=\left(\tfrac{L}{n}\right)^{k}\sum_{t\in I}\delta_{t}.
\]
Using the fundamental theorem of calculus one can show that for any $F\in\C(\RR^{k})$ we have the
following bound
\[
\left|\int_{D}F(y)\,dy-\Delta_{n}(F)\right|\leq\frac{kL^{k+1}}{n}p_{D,1}(F),
\]
where $p_{D,1}$ measures the size of the gradient of $F$ and is defined as in Subsection \ref{Subsection Locally convex spaces}.
This bound basically says that the sequence $(\Delta_{n})_{n\in\NN}$ converges to $\int_{D}$ in $\cE'(\RR^{k})$.

\begin{prop}\label{Proposition Denseness of Dirac distributions}
The space $\Dirac_{\pi}(\RR^{l}\times\RR^{k})$ is a dense subspace of $\Cc(\RR^{l},\cE'(\RR^{k}))$.
\end{prop}
\begin{proof} 
By Proposition \ref{Smooth functions are dense in the space of distributions} it is enough to show that 
for every $\phi\in\Cc(\RR^{l}\times\RR^{k})$ the distribution
$T_{\phi}\in\Cc(\RR^{l},\cE'(\RR^{k}))$ can be approximated arbitrarily well by elements of $\Dirac_{\pi}(\RR^{l}\times\RR^{k})$.

Choose any $\phi\in\Cc(\RR^{l}\times\RR^{k})$ and suppose $\pi_{\RR^{k}}(\supp(\phi))\subset D\subset\RR^{k}$ for some $L>0$ as above.
For $n\in\NN$ we define a $\pi$-transversal distribution $\Delta_{\phi,n}\in\Dirac_{\pi}(\RR^{l}\times\RR^{k})$ by the formula
\[
\Delta_{\phi,n}(F)(x)=\left(\tfrac{L}{n}\right)^{k}\sum_{t\in I}\phi(x,t)F(x,t).
\]
We will show that $\Delta_{\phi,n}\to T_{\phi}$ in $\Cc(\RR^{l},\cE'(\RR^{k}))$ as $n\to\infty$.
This means that for every set of the form $V_{\textbf{B},\textbf{m},\textbf{e}}\subset\Cc(\RR^{l},\cE'(\RR^{k}))$ 
we have $T_{\phi}-\Delta_{\phi,n}\in V_{\textbf{B},\textbf{m},\textbf{e}}$ for $n\in\NN$ big enough.
Both $T_{\phi}$ and $\Delta_{\phi,n}$ have supports contained in $\pi(\supp(\phi))\subset K_{j}$ for some $j\in\NN$, so we have to show
that 
\[
|D^{\alpha}_{x}(T_{\phi}(F)-\Delta_{\phi,n}(F))(x)|<\epsilon_{j}
\]
for $x\in K_{j}$, $F\in B_{j}$ and $|\alpha|\leq m_{j}$. Since $B_{j}$ is a bounded subset of $\C(\RR^{l}\times\RR^{k})$, the set
$\tilde{B}_{j}=\phi B_{j}=\{\tilde{F}=\phi F\,|\,F\in B_{j}\}$ is bounded in $\C(\RR^{l}\times\RR^{k})$ as well, so there exists
a constant $C<\infty$ such that $\sup\{p_{K_{j}\times D,m_{j}+1}(\tilde{F})\,|\,\tilde{F}\in\tilde{B}_{j}\}<C$. For $F\in B_{j}$, $x\in K_{j}$
and $|\alpha|\leq m_{j}$ we now compute:
\begin{align*}
|D^{\alpha}_{x}(T_{\phi}(F)-\Delta_{\phi,n}(F))(x)|
&=\left|\int_{D}D^{\alpha}_{x}(\phi(x,y)F(x,y))\,dy-\left(\tfrac{L}{n}\right)^{k}\sum_{t\in I}D^{\alpha}_{x}(\phi(x,t)F(x,t))\right|, \\
&=\left|\int_{D}(D^{\alpha}_{x}\tilde{F})(x,y)\,dy-\left(\tfrac{L}{n}\right)^{k}\sum_{t\in I}(D^{\alpha}_{x}\tilde{F})(x,t)\right|, \\
&\leq \frac{kL^{k+1}}{n}p_{K_{j}\times D,m_{j}+1}(\tilde{F})<\frac{CkL^{k+1}}{n}.
\end{align*}
We conclude that $\Delta_{\phi,n}\to T_{\phi}$ in $\Cc(\RR^{l},\cE'(\RR^{k}))$ as $n\to\infty$.
\end{proof}

As a corollary we get the following result.

\begin{prop}
Let $M\times N$ be a trivial bundle over $M$ with fiber $N$ and projection $\pi:M\times N\to M$. 
The space $\Dirac_{\pi}(M\times N)$ is a dense subspace of $\cE'_{\pi}(M\times N)$.
\end{prop}
\begin{proof} 
Every transversal distribution $T\in\cE'_{\pi}(M\times N)$ has compact support, so we can write it as a sum
\[
T=T_{1}+T_{2}+\ldots+T_{n},
\]
where each $T_{i}\in\cE'_{\pi}(M\times N)$ has support contained in the set of the form $U_{i}\times U_{i}'$ for
some domains of coordinate charts $U_{i}\approx\RR^{l}$ on $M$ and $U_{i}'\approx\RR^{k}$ on $N$. By Proposition 
\ref{Proposition Denseness of Dirac distributions} we can find for each neighbourhood
of zero $V\subset\cE'_{\pi}(M\times N)$ elements $u_{i}\in\Dirac_{\pi_{i}}(U_{i}\times U_{i}')\subset\Dirac_{\pi}(M\times N)$,
such that $T_{i}-u_{i}\in\frac{1}{n}V$ for $1\leq i\leq n$. 
If we define $u=u_{1}+u_{2}+\ldots+u_{n}\in\Dirac_{\pi}(M\times N)$, we then have $T-u\in V$.
\end{proof}

Let us now denote for simplicity by:
\begin{align*}
\Dirac_{\pi}(M\times N)'&=\Hom_{\Cc(M)}(\Dirac_{\pi}(M\times N),\Cc(M)), \\
\cE'_{\pi}(M\times N)'&=\Hom_{\Cc(M)}(\cE'_{\pi}(M\times N),\Cc(M))
\end{align*}
the $\Cc(M)$-duals of the $\Cc(M)$-modules $\Dirac_{\pi}(M\times N)$ and $\cE'_{\pi}(M\times N)$. Define
a $\Cc(M)$-linear map
\[
\hat{}:\C(M\times N)\to\Dirac_{\pi}(M\times N)',
\]
by 
\[
\hat{F}(u)=u(F)
\] 
for $F\in\C(M\times N)$ and $u\in\Dirac_{\pi}(M\times N)$.

\begin{theo}\label{Strong dual of distributions of Dirac type}
Let $M\times N$ be a trivial bundle over $M$ with fiber $N$ and bundle projection $\pi:M\times N\to M$. 
The map $\hat{}:\C(M\times N)\to\Dirac_{\pi}(M\times N)'$ is an isomorphism of locally convex $\Cc(M)$-modules.
\end{theo}
\begin{proof}
We first show that the map $\hat{}:\C(M\times N)\to\Dirac_{\pi}(M\times N)'$ is a $\Cc(M)$-linear isomorphism. It is injective
since $\Dirac_{\pi}(M\times N)$ separates the points of $\C(M\times N)$. To see that it is surjective, choose any
$\phi\in\Dirac_{\pi}(M\times N)'$. Since $\Dirac_{\pi}(M\times N)$ is a dense subspace of $\cE'_{\pi}(M\times N)$
and since $\Cc(M)$ is complete, there exists a unique continuous extension $\overline{\phi}:\cE'_{\pi}(M\times N)\to\Cc(M)$
of $\phi$ to $\cE'_{\pi}(M\times N)$. From Theorem $4.5$ in \cite{Kal22_2} it now follows that $\overline{\phi}=\hat{F}$
for some $F\in\C(M\times N)$.

It remains to be shown that the map $\hat{}:\C(M\times N)\to\Dirac_{\pi}(M\times N)'$ is a homeomorphism. It is
continuous as it can be written as a composition
\[
\C(M\times N)\overset{\hat{}}{\longrightarrow}\cE'_{\pi}(M\times N)'\longrightarrow\Dirac_{\pi}(M\times N)',
\]
where the left map is continuous by Theorem $4.5$ in \cite{Kal22_2} and the right map is the continuous restriction of functionals from
$\cE'_{\pi}(M\times N)$ to $\Dirac_{\pi}(M\times N)$.

In the remainder of the proof we will show that the above map is open. Let us choose an arbitrary subbasic neighbourhood
of zero in $\C(M\times N)$ of the form
\[
V_{L\times K,m,\epsilon}=\{F\in\C(M\times N)\,|\,|D^{\alpha}_{x}D^{\beta}_{y}F(x,y)|<\epsilon\text{ for }(x,y)\in L\times K,\,|\alpha|+|\beta|\leq m\},
\]
where $m\in\NN$, $\epsilon>0$, $L$ is a compact subset of $M$ which lies in some chart $U_{M}\approx\RR^{l}$ and $K$ is a compact subset of $N$ which
lies in some chart $U_{N}\approx\RR^{k}$. Our goal is to find a bounded subset $B\subset\Dirac_{\pi}(M\times N)$ and a
neighbourhood $V$ of zero in $\Cc(M)$ such that $K(B,V)\subset\widehat{V_{L\times K,m,\epsilon}}$.

For $n\in\NN$, $t\in(0,\infty)$ and $y\in\RR$ we define a distribution $\Delta_{t}^{n}(y)\in\Dirac(\RR)$ by
\[
\Delta_{t}^{n}(y)=\frac{1}{(2t)^{n}}\sum_{k=0}^{n}(-1)^{k}\binom{n}{k}\delta_{y+(n-2k)t}.
\]
Using the Taylor's theorem one can show that $\Delta_{t}^{n}(y)$ converges in $\cE'(\RR)$ to $D^{n}_{y}|_{y}$ as $t\to 0$,
where $D^{n}_{y}|_{y}$ is the distribution which computes the $n$-th derivative at the point $y$.
More generally, denote $\beta=(\beta_{1},\beta_{2},\ldots,\beta_{k})\in\NN_{0}^{k}$, 
$y=(y_{1},y_{2},\ldots,y_{k})\in\RR^{k}$ and define
\[
\Delta_{t}^{\beta}(y)=\Delta_{t}^{\beta_{1}}(y_{1})\otimes\Delta_{t}^{\beta_{2}}(y_{2})\otimes\cdots\otimes\Delta_{t}^{\beta_{k}}(y_{k})\in\cE'(\RR^{k}).
\]
Again we have that $\Delta_{t}^{\beta}(y)\in\Dirac(\RR^{k})$ converges to $D^{\beta}_{y}|_{y}$ in $\cE'(\RR^{k})$ as $t\to 0$. Using $K$ and $m$ from the
definition of $V_{L\times K,m,\epsilon}$ we now define the subset
\[
B_{K,m}=\{\Delta_{t}^{\beta}(y)\,|\,y\in K, t\in(0,1),|\beta|\leq m\}\subset\Dirac(\RR^{k})\subset\cE'(\RR^{k}).
\]
Using estimates from the Taylor's theorem one can show that $B_{K,m}$ is a bounded subset of $\cE'(\RR^{k})$.
Now note that the bilinear map $\Cc(M)\times\cE'(N)\to\Cc(M,\cE'(N))$, given by $(f,v)\mapsto fv$ for $(fv)_{x}=f(x)v$, is
continuous. If we choose a function $\eta\in\Cc(U_{M})\subset\Cc(M)$, such that $\eta\equiv 1$ on some neighbourhood of $L$, 
it now follows from the above observation that
\[
B=\eta B_{K,m}=\{\eta v\,|\,v\in B_{K,m}\}
\]
is a bounded subset of $\Dirac_{\pi}(M\times N)$. Finally, let us define an open neighbourhood $V$ of zero in $\Cc(M)$ by
\[
V=\{f\in\Cc(M)\,|\,|D^{\alpha}_{x}f(x)|<\tfrac{\epsilon}{2}\text{ for }x\in L,\,|\alpha|\leq m\}.
\]

Now choose any $\phi=\hat{F}\in K(B,V)$ so that $\phi(u)=u(F)\in V$ for $u\in B$. If we write $u=\eta v=\eta\sum_{i=1}^{n}a_{i}\delta_{y_{i}}$
for some $a_{1},\ldots,a_{n}\in\CC$ and some $y_{1},\ldots,y_{n}\in U_{N}$, we have for $x\in L$ and $|\alpha|\leq m$ the following estimate
\[
|u(D^{\alpha}_{x}F)(x)|=|\eta(x)\sum_{i=1}^{n}a_{i}(D^{\alpha}_{x}F)(x,y_{i})|
=|D^{\alpha}_{x}(u(F))(x)|<\tfrac{\epsilon}{2}.
\]
Here we have used the fact that $\eta\equiv 1$ on some neighbourhood of $L$ and denoted by $D^{\alpha}_{x}F$ 
the $\alpha$-partial derivative of $F$ in the horizontal direction. For any $y\in K$ and any $\beta$ with $|\beta|\leq m$ the net
$\Delta_{t}^{\beta}(y)\in B_{K,m}$ converges to $D^{\beta}_{y}|_{y}$ in $\cE'(\RR^{k})$ as $t\to 0$. 
If we define $u_{t}=\eta \Delta_{t}^{\beta}(y)\in B$, we then have for $x\in L$
the estimate
\[
|D^{\beta}_{y}D^{\alpha}_{x}F(x,y)|=\lim_{t\to 0}|u_{t}(D^{\alpha}_{x}F)(x)|\leq\tfrac{\epsilon}{2}<\epsilon.
\]
To sum it up, for $(x,y)\in L\times K$ and $|\alpha|,|\beta|\leq m$ we have $|D^{\beta}_{y}D^{\alpha}_{x}F(x,y)|<\epsilon$, which
implies that $F\in V_{L\times K,m,\epsilon}$ and consequently $\phi=\hat{F}\in\widehat{V_{L\times K,m,\epsilon}}$.
\end{proof}

\section{Spectral bundle of the coalgebra of transversal distributions of constant Dirac type}\label{Section Spectral bundle of coalgebra of Dirac distributions}

From the Theorem \ref{Strong dual of distributions of Dirac type} it follows that $\C(M\times N)$ is
isomorphic to the strong $\Cc(M)$-dual of $\Dirac_{\pi}(M\times N)$. We will now equip the space $\Dirac_{\pi}(M\times N)$
with a structure of a locally convex coalgebra over $\Cc(M)$, such that its strong $\Cc(M)$-dual $\Dirac_{\pi}(M\times N)'$
is a Fr\'{e}chet algebra, isomorphic to $\C(M\times N)$.

We will use the isomorphism $\Psi_{M\times N}:\Cc(M\times\Nd)\to\Dirac_{\pi}(M\times N)$
to transfer coalgebra structure from $\Cc(M\times\Nd)$ to $\Dirac_{\pi}(M\times N)$. Explicitly, using the
notation from the Example \ref{Example sheaf coalgebras}, we define
on $\Dirac_{\pi}(M\times N)$ a structure of a coalgebra over $\Cc(M)$ with structure maps:
\begin{align*}
\cm&:\Dirac_{\pi}(M\times N)\to\Dirac_{\pi}(M\times N)\otimes_{\Cc(M)}\Dirac_{\pi}(M\times N), \\
\cu&:\Dirac_{\pi}(M\times N)\to\Cc(M),
\end{align*}
explicitly given by:
\begin{align*}
\cm(\sum_{i=1}^{n}\llbracket E_{y_{i}},f_{i}\rrbracket)&=\sum_{i=1}^{n}\llbracket E_{y_{i}},f_{i}\rrbracket\otimes\llbracket E_{y_{i}},1_{f_{i}}\rrbracket, \\
\cu(\sum_{i=1}^{n}\llbracket E_{y_{i}},f_{i}\rrbracket)&=\sum_{i=1}^{n}f_{i}.
\end{align*}

Since $\C(M\times N)$ is isomorphic to the strong dual of $\Dirac_{\pi}(M\times N)$, we can use it to define
the $\Cc(M)$-injective topology on $\Dirac_{\pi}(M\times N)\otimes_{\Cc(M)}\Dirac_{\pi}(M\times N)$. 
For any pair of functions $F,G\in\C(M\times N)$ we define a $\Cc(M)$-linear map 
\[
F\otimes G:\Dirac_{\pi}(M\times N)\otimes_{\Cc(M)}\Dirac_{\pi}(M\times N)\to\Cc(M)
\]
by
\[
(F\otimes G)(\sum_{i=1}^{n}u_{i}'\otimes u_{i}'')=\sum_{i=1}^{n}u_{i}'(F)u_{i}''(G).
\]
The $\Cc(M)$-injective topology on $\Dirac_{\pi}(M\times N)\otimes_{\Cc(M)}\Dirac_{\pi}(M\times N)$ 
is now defined by specifying basic neighbourhoods of zero of the form
\[
K(A,B,V)=\{\tilde{u}\in\Dirac_{\pi}(M\times N)^{\otimes 2}\,|\,(F\otimes G)(\tilde{u})\in V,\text{ for}\,F\in A,\,G\in B\},
\]
where $A,B\subset\C(M\times N)\cong\Dirac_{\pi}(M\times N)'$ are bounded subsets and $V$ is a neighbourhood of zero in $\Cc(M)$.

\begin{prop}\label{Proposition Dirac is locally convex coalgebra}
The triple $(\Dirac_{\pi}(M\times N),\cm,\cu)$ is a cocommutative, locally
convex coalgebra over $\Cc(M)$, in the sense that $\cm$ and $\cu$ are continuous maps.
\end{prop}
\begin{proof} 
Let us denote by $1$ the unit of the algebra $\C(M\times N)$. We then have $\cu=\hat{1}$, which shows that $\cu$
is a continuous map. 

To see that $\cm$ is continuous, we choose any basic neighbourhood of zero in $\Dirac_{\pi}(M\times N)\otimes_{\Cc(M)}\Dirac_{\pi}(M\times N)$
of the form $K(A,B,V)$ as above. The set 
\[
A\cdot B=\{FG\,|\,F\in A,\,G\in B\}
\] 
is then a bounded subset of $\C(M\times N)$. For any 
$u=\sum_{i=1}^{n}\llbracket E_{y_{i}},f_{i}\rrbracket\in K(A\cdot B,V)$ any $F\in A$ and any $G\in B$ we now have
\begin{align*}
(F\otimes G)(\cm(u))(x)&=(F\otimes G)(\sum_{i=1}^{n}\llbracket E_{y_{i}},f_{i}\rrbracket\otimes\llbracket E_{y_{i}},1_{f_{i}}\rrbracket)(x), \\
&=\sum_{i=1}^{n}f_{i}(x)1_{f_{i}}(x)F(x,y_{i})G(x,y_{i}), \\
&=u(FG)(x).
\end{align*}
This implies that $\cm(K(A\cdot B,V))\subset K(A,B,V)$ hence $\cm$ is continuous.
\end{proof}

Since $\Dirac_{\pi}(M\times N)$ is a cocommutative, counital coalgebra over $\Cc(M)$, its dual $\Dirac_{\pi}(M\times N)'$ naturally becomes 
a commutative algebra with unit $\cu$ over $\Cc(M)$, if we define
\[
(\phi\cdot\psi)(u)=(\phi\otimes\psi)(\cm(u))
\]
for $\phi,\psi\in\Dirac_{\pi}(M\times N)'$ and $u\in\Dirac_{\pi}(M\times N)$. Continuity of $\phi\cdot\psi$
follows from continuity of $\phi\otimes\psi$ and $\cm$. 

On both $\Dirac_{\pi}(M\times N)$ and $\Dirac_{\pi}(M\times N)'$ we can naturally define conjugation as follows.
For any $u\in\Dirac_{\pi}(M\times N)$ we define conjugation by
\[
u=\sum_{i=1}^{n}\llbracket E_{y_{i}},f_{i}\rrbracket\Longrightarrow \overline{u}=\sum_{i=1}^{n}\llbracket E_{y_{i}},\overline{f_{i}}\rrbracket.
\]
Using the above formula and complex conjugation on $\Cc(M)$ we now define for any $\phi\in\Dirac_{\pi}(M\times N)'$ the element 
$\overline{\phi}\in\Dirac_{\pi}(M\times N)'$ by
\[
\overline{\phi}(u)=\overline{\phi(\overline{u})}
\]
for $u\in\Dirac_{\pi}(M\times N)$. It is now a straightforward calculation to extend the Theorem \ref{Strong dual of distributions of Dirac type}
in the following way.

\begin{prop}\label{Isomorphism of locally convex algebras}
Let $M\times N$ be a trivial bundle over $M$ with fiber $N$ and bundle projection $\pi:M\times N\to M$. 
The map $\hat{}:\C(M\times N)\to\Dirac_{\pi}(M\times N)'$ is an isomorphism of locally convex algebras with involutions.
\end{prop}

Using the definitions and notations from the Subsection \ref{Subsection Coalgebras} we now define for any $x\in M$ 
the local $\Cc(M)_{x}$-coalgebra
\[
\Dirac_{\pi}(M\times N)_{x}=\Dirac_{\pi}(M\times N)/I_{x}\Dirac_{\pi}(M\times N).
\]
It follows from \cite{Mrc07_2} that the space
$\Dirac_{\pi}(M\times N)_{x}$ is a free $\Cc(M)_{x}$-module, generated by the set $G(\Dirac_{\pi}(M\times N)_{x})$ of grouplike elements.
The spectral sheaf of the $\Cc(M)$-coalgebra $\Dirac_{\pi}(M\times N)$ is the sheaf
\[
\pisp:\Esp(\Dirac_{\pi}(M\times N))\to M
\]
with the stalk at the point $x\in M$ given by
\[
\Esp(\Dirac_{\pi}(M\times N))_{x}=G(\Dirac_{\pi}(M\times N)_{x}).
\]
Note that the sheaves $M\times\Nd$ and $\Esp(\Dirac_{\pi}(M\times N))$ over $M$ are isomorphic via the map
\[
(x,y)\mapsto\llbracket E_{y},f\rrbracket|_{x},
\]
where $f\in\Cc(M)$ is any function with $f|_{x}=1\in\Cc(M)_{x}$.

Let us now define the real part of $\Dirac_{\pi}(M\times N)'$ by
\[
\Dirac_{\pi}(M\times N)'_{\RR}=\{\phi\in\Dirac_{\pi}(M\times N)'\,|\,\overline{\phi}=\phi\}
\]
and note that it corresponds to the algebra $\C(M\times N,\RR)$ via 
the isomorphism from Proposition \ref{Isomorphism of locally convex algebras}. 
This implies that $\Dirac_{\pi}(M\times N)'_{\RR}$ satisfies the
conditions of the main theorem in \cite{MiVa96}, so it can be used to define a smooth structure on the space
$\Spec(\Dirac_{\pi}(M\times N)'_{\RR})$. Furthermore, we have a natural bijection
\[
\Theta_{M\times N}:\Esp(\Dirac_{\pi}(M\times N))\to\Spec(\Dirac_{\pi}(M\times N)'_{\RR}),
\]
defined by
\[
\Theta_{M\times N}(\llbracket E_{y},f\rrbracket|_{x})(\phi)=\phi(\llbracket E_{y},f\rrbracket)(x)
\]
for $\phi\in\Dirac_{\pi}(M\times N)'_{\RR}$. We will now use this bijection to transfer the smooth structure
from $\Spec(\Dirac_{\pi}(M\times N)'_{\RR})$ to $\Esp(\Dirac_{\pi}(M\times N))$.

\begin{dfn}
Let $\pi:M\times N\to M$ be a trivial bundle over $M$ with fiber $N$.
\textbf{The spectral bundle} 
\[
\Bsp(\Dirac_{\pi}(M\times N))
\] 
of the coalgebra $\Dirac_{\pi}(M\times N)$ is the set $\Esp(\Dirac_{\pi}(M\times N))$, equipped with 
the bundle projection $\pisp:\Bsp(\Dirac_{\pi}(M\times N))\to M$ and the topology and smooth structure 
such that $\Theta_{M\times N}$ is a diffeomorphism.
\end{dfn}

We will show in the next theorem that $\Bsp(\Dirac_{\pi}(M\times N))$ is a trivial bundle over $M$, naturally isomorphic to 
the bundle $M\times N$. Define a map
\[
\Phi^{\text{bun}}_{M\times N}:M\times N\to\Bsp(\Dirac_{\pi}(M\times N)),
\] 
by
\[
\Phi^{\text{bun}}_{M\times N}(x,y)=\llbracket E_{y},f\rrbracket|_{x}.
\]

\begin{theo}\label{Theorem Isomorphism of trivial bundles}
Let $\pi:M\times N\to M$ be a trivial bundle over $M$ with fiber $N$.
The map $\Phi^{\text{bun}}_{M\times N}:M\times N\to\Bsp(\Dirac_{\pi}(M\times N))$  
is an isomorphism of trivial bundles over $M$.
\end{theo}
\begin{proof}
Let us denote by $\Sp:\Spec(\C(M\times N,\RR))\to\Spec(\Dirac_{\pi}(M\times N)'_{\RR})$ the diffeomorphism, induced
by the inverse of $\hat{}:\C(M\times N,\RR)\to\Dirac_{\pi}(M\times N)'_{\RR}$. We then have the commutative diagram
\[
\begin{CD}
M\times N @>\Phi^{\text{bun}}_{M\times N}>>\Bsp(\Dirac_{\pi}(M\times N))  \\
@V\Phi^{\text{man}}_{M\times N}VV @VV\Theta_{M\times N}V \\
\Spec(\C(M\times N,\RR)) @>\Sp>>\Spec(\Dirac_{\pi}(M\times N)'_{\RR})
\end{CD}
\]
Since $\Phi^{\text{man}}_{M\times N}$, $\Theta_{M\times N}$ and $\Sp$ are diffeomorphisms, $\Phi^{\text{bun}}_{M\times N}$
is a diffeomorphism as well.
The equality $\pisp\com\Phi^{\text{bun}}_{M\times N}=\pi$ follows from the equality $\pisp\com\Phi^{\text{shv}}_{M\times\Nd}=\pi$.
\end{proof}

Let us now take a look at this construction in the case of a single manifold.

\begin{ex} \rm \label{Example Dirac(N)}
Let $M$ be a single point and consider the manifold $N$ as a trivial bundle over a point. In this case we have
\[
\Dirac_{\pi}(M\times N)=\Dirac(N)=\text{Span}\{\delta_{y}\,|\,y\in N\}.
\]
Every element $u\in\Dirac(N)$ can be expressed as a finite sum $u=\sum_{i=1}^{n}\lambda_{i}\delta_{y_{i}}$
for unique $\lambda_{1},\ldots,\lambda_{n}\in\CC$ and $y_{1},\ldots,y_{n}\in N$. The space $\Dirac(N)$ is 
a coalgebra over $\CC$ with structure maps:
\begin{align*}
\cm(\sum_{i=1}^{n}\lambda_{i}\delta_{y_{i}})&=\sum_{i=1}^{n}\lambda_{i}\delta_{y_{i}}\otimes\delta_{y_{i}}, \\
\cu(\sum_{i=1}^{n}\lambda_{i}\delta_{y_{i}})&=\sum_{i=1}^{n}\lambda_{i}.
\end{align*}
Grouplike elements of $\Dirac(N)$ are precisely Dirac distributions, so we have 
\[
G(\Dirac(N))=\{\delta_{y}\,|\,y\in N\}.
\]
The spectral sheaf $\Esp(\Dirac(N))$ is the set $G(\Dirac(N))$ with the discrete topology and projection
onto the point. The map $\Theta_{N}:\Esp(\Dirac(N))\to\Spec(\Dirac(N)'_{\RR})$ is defined by
$
\Theta_{N}(\delta_{y})(\phi)=\phi(\delta_{y}),
$
which means that $\Theta_{N}(\delta_{y})=\hat{\delta}_{y}$. The topology on $\Bsp(\Dirac(N))=G(\Dirac(N))$, which is induced
by $\Theta_{N}$, coincides with the subspace topology on $G(\Dirac(N))$, induced from $\cE'(N)$.
Finally, the diffeomorphism
\[
\Phi^{\text{bun}}_{N}:N\to\Bsp(\Dirac(N))
\] 
is given by $\Phi^{\text{bun}}_{N}(y)=\delta_{y}$. 
\end{ex}

\section{Locally convex bialgebroid of an action Lie groupoid}\label{Section Locally convex bialgebroid}

In this section we will assign to each action groupoid $M\rtimes H$ 
a locally convex bialgebroid with antipode $\Dirac(M\rtimes H)$ over $\Cc(M)$, from which 
the Lie groupoid $M\rtimes H$ can be reconstructed.

Let $M$ be a second-countable manifold and let $H$ be a second-countable Lie group, which acts on $M$
from the right. If we denote by $\Hd$ the group $H$ with the discrete topology, the group $\Hd$ acts on $M$
from the right as well, so we obtain two action groupoids $M\rtimes H$ and $M\rtimes\Hd$. 
These two groupoids are isomorphic as groupoids but not as Lie groupoids if $\dim(H)>0$. 

Groupoid $M\rtimes\Hd$ is \'{e}tale, so we can construct its Hopf algebroid
\[
\Cc(M\rtimes\Hd)=\bigoplus_{h\in\Hd}\Cc(M\times\{h\}).
\]
Moreover, since $M\rtimes H$ is a Lie groupoid, we also have a convolution product, as defined in \cite{LeMaVa17}, on the space
\[
\cE'_{\trg}(M\rtimes H)=\Hom_{\Cc(M)}(\C(M\times H),\Cc(M))
\]
of $\trg$-transversal distributions on $M\rtimes H$. It can be described explicitly as follows. 
The left translation by $g\in M\rtimes H$ is
the diffeomorphism $\Lt_{g}:\trg^{-1}(\src(g))\to\trg^{-1}(\trg(g))$, defined by $\Lt_{g}(h)=gh$.
For any $F\in\C(M\times H)$ and any $g\in M\rtimes H$ it follows that 
$F\com \Lt_{g}\in\C(\trg^{-1}(\src(g)))$ and one can show that the function
$M\times H\to\RR$, $g\mapsto T_{\src(g)}(F\com \Lt_{g})$,
is smooth for any $T\in\cE'_{\trg}(M\rtimes H)$. For any $T',T''\in\cE'_{\trg}(M\rtimes H)$
the convolution $T'\ast T''\in \cE'_{\trg}(M\rtimes H)$ is then defined by
\[
(T'\ast T'')(F)(x)=
  T'\left(g\mapsto T''_{\src(g)}(F\com \Lt_{g})\right)(x),
\]
for any $F\in\C(M\times H)$ and any $x\in M$.

Using the notation from the Example \ref{Example sheaf coalgebras} we define an injective 
$\Cc(M)$-linear map $\Psi_{M\rtimes H}:\Cc(M\rtimes\Hd)\to\cE'_{\trg}(M\rtimes H)$ by
\[
\Psi_{M\rtimes H}\big(\sum_{i=1}^{n}f_{i}\cdot\delta_{h_{i}}\big)=\sum_{i=1}^{n}\llbracket E_{h_{i}},f_{i}\rrbracket.
\]

\begin{prop}
The map $\Psi_{M\rtimes H}:\Cc(M\rtimes\Hd)\to\cE'_{\trg}(M\rtimes H)$ is multiplicative.
\end{prop}
\begin{proof}
Let us first prove that $\Psi_{M\rtimes H}$ is multiplicative on the set of basis elements. Choose $f_{1},f_{2}\in\Cc(M)$,
$h_{1},h_{2}\in H$ and denote $a_{1}=f_{1}\cdot\delta_{h_{1}}$ respectively $a_{2}=f_{2}\cdot\delta_{h_{2}}$.
We then have:
\begin{align*}
(\Psi_{M\rtimes H}(a_{1})\ast\Psi_{M\rtimes H}(a_{2}))(F)(x)
&=\Psi_{M\rtimes H}(a_{1})\left((x,h)\mapsto \llbracket E_{h_{2}},f_{2}\rrbracket_{xh}(F\com\Lt_{(x,h)})\right)(x), \\
&=\llbracket E_{h_{1}},f_{1}\rrbracket\left((x,h)\mapsto f_{2}(xh)F(x,hh_{2})\right)(x), \\
&=f_{1}(x)f_{2}(xh_{1})F(x,h_{1}h_{2}).
\end{align*}
On the other hand (see Example \ref{Example Locally grouplike Hopf algebroid}) we have $a_{1}*a_{2}=(f_{1}(h_{1}f_{2}))\cdot\delta_{h_{1}h_{2}}$
hence
\[
\Psi_{M\rtimes H}(a_{1}*a_{2})(F)(x)=\llbracket E_{h_{1}h_{2}},f_{1}(h_{1}f_{2})\rrbracket(F)(x)
=f_{1}(x)f_{2}(xh_{1})F(x,h_{1}h_{2}).
\]
Multiplicativity of the map $\Psi_{M\rtimes H}$ now follows from linearity of $\Psi_{M\rtimes H}$ and bilinearity of both
convolution products.
\end{proof}

\begin{dfn}
Let $M\rtimes H$ be an action groupoid of an action of a second-countable Lie group $H$ on a second-countable
manifold $M$ and let $M\rtimes\Hd$ be the assoicated \'{e}tale groupoid. The \textbf{Dirac bialgebroid} of $M\rtimes H$ is the space
\[
\Dirac(M\rtimes H)=\Psi_{M\rtimes H}(\Cc(M\rtimes\Hd)).
\]
\end{dfn}

The Dirac bialgebroid $\Dirac(M\rtimes H)$ inherits from $\Cc(M\rtimes\Hd)$ a structure of a locally grouplike Hopf algebroid over $\Cc(M)$.
Moreover, by Proposition \ref{Proposition Dirac is locally convex coalgebra} we obtain on $\Dirac(M\rtimes H)$ a structure of 
a locally convex coalgebra. Finally, as shown in \cite{LeMaVa17}, the multiplication on $\cE'_{\trg}(M\rtimes H)$
and hence on $\Dirac(M\rtimes H)$ is separately continuous. We sum up these observations in the following proposition.

\begin{prop}\label{Proposition Dirac is locally convex bialgebroid}
The Dirac bialgebroid $\Dirac(M\rtimes H)$ of any action Lie groupoid $M\rtimes H$ is a locally convex bialgebroid with an antipode over 
$\Cc(M)$. 
\end{prop}

\begin{rem}\rm
A locally convex bialgebroid is a bialgebroid $(A,\cm,\cu,\mu)$, equipped with a locally convex structure such that
$\cm$ and $\cu$ are continuous maps and $\mu$ is separately continuous. We do not know if the antipode $S$
on $\Dirac(M\rtimes H)$ is continuous in general, which would mean that it is a locally convex Hopf algebroid.
\end{rem}

\begin{ex} \rm
Let us take a look at the case when the group $H$ acts trivially on $M$. The associated action
groupoid $M\rtimes H$ is in this case just the trivial bundle of Lie groups over $M$ with fiber $H$, which
will be denoted by $M\times H$. The multiplication and antipode can be expressed on generators by the formulas:
\begin{align*}
\llbracket E_{h_{1}},f_{1}\rrbracket*\llbracket E_{h_{2}},f_{2}\rrbracket&=\llbracket E_{h_{1}h_{2}},f_{1}f_{2}\rrbracket, \\
S(\llbracket E_{h},f\rrbracket)&=\llbracket E_{h^{-1}},f\rrbracket.
\end{align*}
In this case $\Cc(M)$ is a central subalgebra of $\Dirac(M\times H)$.
Moreover, from the equality $\trg\com\inv=t$ it follows that $S$ is continuous. As a result we see that $\Dirac(M\times H)$
is a locally convex Hopf algebra over $\Cc(M)$. 
\end{ex}

Now take any action Lie groupoid $M\rtimes H$. The spectral
\'{e}tale Lie groupoid $\Gsp(\Dirac(M\rtimes H))$ of $\Dirac(M\rtimes H)$ is then isomorphic to the \'{e}tale groupoid $M\rtimes\Hd$.
Moreover, since $\Dirac(M\rtimes H)$ is a locally convex coalgebra over $\Cc(M)$, we also have the 
bijection $\Theta_{M\rtimes H}:\Gsp(\Dirac(M\rtimes H))\to\Spec(\Dirac(M\rtimes H))'_{\RR})$.

\begin{dfn}
\textbf{The spectral action Lie groupoid}   
\[
\AGsp(\Dirac(M\rtimes H))
\] 
of the Dirac bialgebroid $\Dirac(M\rtimes H)$ is the groupoid $\Gsp(\Dirac(M\rtimes H))$, equipped
with the smooth structure such that the map $\Theta_{M\rtimes H}$ is a diffeomorphism.
\end{dfn}

Define a map
\[
\Phi^{\text{agr}}_{M\rtimes H}:M\rtimes H\to\AGsp(\Dirac(M\rtimes H)),
\] 
by
\[
\Phi^{\text{agr}}_{M\rtimes H}(x,h)=\llbracket E_{h},f\rrbracket|_{x},
\]
where $f\in\Cc(M)$ is such that $f|_{x}=1\in\Cc(M)_{x}$.

\begin{theo}
Let $M\rtimes H$ be an action groupoid of an action of a second-countable Lie group $H$
on a second-countable manifold $M$. The map 
\[
\Phi^{\text{agr}}_{M\rtimes H}:M\rtimes H\to\AGsp(\Dirac(M\rtimes H))
\]
is an isomorphism of Lie groupoids.
\end{theo}
\begin{proof}
Since $M\rtimes H$ is isomorphic to $M\rtimes\Hd$ and $\AGsp(\Dirac(M\rtimes H))$ is isomorphic to
$\Gsp(\Dirac(M\rtimes H))$, the map $\Phi^{\text{agr}}_{M\rtimes H}$ is an isomorphism of groupoids.
By Theorem \ref{Theorem Isomorphism of trivial bundles} it is also a diffeomorphism, 
which implies that it is an isomorphism of Lie groupoids.
\end{proof}

\begin{ex} \rm
Let $H$ be a Lie group so that $\Dirac(H)$ is a locally convex Hopf algebra over $\CC$.
Example \ref{Example Dirac(N)} shows that $\AGsp(\Dirac(H))=\{\delta_{h}\,|\,h\in H\}$ is
naturally diffeomorphic to $H$. The multiplication and inverse maps on $\AGsp(\Dirac(H))$ are induced 
by the multiplication and the antipode on $\Dirac(H)$. Namely, for any $h,h'\in H$ we have:
\begin{align*}
\delta_{h}\delta_{h'}&=\delta_{h}*\delta_{h'}=\delta_{hh'}, \\
\delta_{h}^{-1}&=S(\delta_{h})=\delta_{h^{-1}}.
\end{align*}
\end{ex}

\noindent
{\bf Acknowledgements:} 

I would like to thank O. Dragičević, F. Forstnerič, A. Kostenko and J. Mrčun for support during research.


\begin{thebibliography}{99}
\bibitem{AnMoYu21}   I. Androulidakis, O. Mohsen, R. Yuncken,
					 The convolution algebra of Schwarz kernels on a singular foliation.
					 {\em J. Operator Theory} 85(2) (2021) 475--503.

\bibitem{AnSk11}     I. Androulidakis, G. Skandalis,
					 Pseudodifferential calculus on a singular foliation.
					 {\em J. Noncommut. Geom.} 5 (2011) 125--152.
			
\bibitem{Boh09}      G. B\"{o}hm,
                     {\em Hopf algebroids}.
                     Handbook of algebra Vol. 6, Elsevier (2009)
                     173--236.
                     
\bibitem{CSWe99}     A. Cannas da Silva, A. Weinstein,
                     {\em Geometric Models for Noncommutative Algebras}.
                     Berkeley Mathematics Lecture Notes 10,
                     Amer. Math. Soc., Providence, R. I. (1999).
		
\bibitem{Con82}      A. Connes,
                     A survey of foliations and operator algebras.
                     {\em Proc. Sympos. Pure Math.} 38 (1982) 521--628.	 
					 
\bibitem{ErYu19}     E. van Erp, R. Yuncken,
					 A groupoid approach to pseudodifferential operators.
					 {\em J. Reine Angew. Math.} 756 (2019) 151--182.
					 
\bibitem{Hor67}		 J. Horv\'{a}th,
					 {\em Topological Vector Spaces, Vol. I},
					 Addison-Wesley Series in Mathematics,
					 Addison-Wesley Publishing Company XII, (1966).
					 					 
\bibitem{Kal22}      J. Kali\v{s}nik,
                     Automatic continuity of transversal distributions.
                     {\em Proc. Amer. Math. Soc.} 150(12) (2022) 5243--5251.
                     
\bibitem{Kal22_2}    J. Kali\v{s}nik,
                     Reflexivity of the space of transversal distributions.
                     https://arxiv.org/abs/2211.07257.
                     
\bibitem{KaMr13}     J. Kali\v{s}nik, J. Mr\v{c}un,
                     A Cartier-Gabriel-Kostant structure theorem for Hopf algebroids.
                     {\em Adv. Math.\/} 232 (2013) 295--310.
                     
\bibitem{KaMr22}     J. Kali\v{s}nik, J. Mr\v{c}un,
					 Convolution bialgebra of a Lie groupoid and transversal distributions.
					 {\em J. Geom. Phys.} 180(104642) (2022).
					 
\bibitem{KaSa19}     L. El Kaoutit, P. Saracco,
					 Topological tensor product of bimodules, complete Hopf algebroids and convolution algebras.
					 {\em Commun. Contemp. Math.} 21 6 (2019).
					 
\bibitem{KoPo09}     N. Kowalzig, H. Posthuma,
                     The cyclic theory of Hopf algebroids.
                     {\em J. Noncommut. Geom.\/} 5 (2011) 423--476.

\bibitem{KrMi97}     A. Kriegl, P. W. Michor,
                     {\em The Convenient Setting of Global Analysis}.
                     Math. Surveys Monogr. 53,
                     Amer. Math. Soc.,
                     Providence, RI (1997).
                     
\bibitem{LeMaVa17}   J.-M. Lescure, D. Manchon, S. Vassout,
                     About the convolution of distributions on groupoids.
                     {\em J. Noncommut. Geom.} 11 (2017) 757--789.
                     
\bibitem{Lu96}       J.-H. Lu,
                     Hopf algebroids and quantum groupoids.
                     {\em Internat. J. Math.\/} 7 (1996) 47--70.
                     
\bibitem{Mac05}      K. C. H. Mackenzie,
                     {\em General Theory of Lie Groupoids
                     and Lie Algebroids}.
                     London Math. Soc. Lecture Note Ser. 213,
                     Cambridge Univ. Press, Cambridge (2005).  
                     
\bibitem{Mal00}      G. Maltsiniotis,
                     Groupo\"{\i}des quantiques de base non commutative.
                     {\em Comm. Algebra\/} 28 (2000) 3441--3501.
                     
\bibitem{MiVa96}     P. W. Michor, J. Van\v{z}ura,
                     Characterizing algebras of smooth functions on manifolds.
                     {\em Comment. Math. Univ. Carolin.} 37 3 (1996) 519--521. 
                     
\bibitem{MoMr03}     I. Moerdijk, J. Mr\v{c}un,
                     {\em Introduction to Foliations and Lie Groupoids}.
                     Cambridge Stud. Adv. Math. 91,
                     Cambridge Univ. Press, Cambridge (2003).

\bibitem{MoMr05}     I. Moerdijk, J. Mr\v{c}un,
                     Lie groupoids, sheaves and cohomology.
                     {\em Poisson Geometry, Deformation Quantisation and Group
                     Representations},
                     London Math. Soc. Lecture Note Ser. 323,
                     Cambridge Univ. Press, Cambridge (2005) 145--272. 
                     
\bibitem{Mrc07_1}    J. Mr\v{c}un,
                     On duality between \'{e}tale groupoids and Hopf algebroids.
                     {\em J. Pure Appl. Algebra} 210 (2007) 267--282.
                     
\bibitem{Mrc07_2}    J. Mr\v{c}un,
                     Sheaf coalgebras and duality.
                     {\em Topology Appl.} 154 (2007) 2795--2812.  
                     
\bibitem{Swe74}      M. E. Sweedler,
                     Groups of simple algebras.
                     {\em Inst. Hautes \'{E}tudes Sci. Publ. Math.}
                     44 (1974) 79--189.
                     
\bibitem{Tak77}      M. Takeuchi,
                     Groups of algebras over $A\otimes \overline A$.
                     {\em J. Math. Soc. Japan},
                     29 (1977) 459--492.                               
                    
\bibitem{Tre67}		 F. Tr\'{e}ves,
				     {\em Topological Vector Spaces, Distributions and Kernels}.
				     Academic Press,
				     New York-London (1967).
				     
\bibitem{Xu98}       P. Xu,
                     Quantum groupoids and deformation quantization.
                     {\em C. R. Math. Acad. Sci. Paris\/} 326 (1998). 
\end{thebibliography}
\end{document}